\documentclass[a4paper,11pt,svgnames,table,xcolor=dvipsnames]{article}
\usepackage[textheight=23cm,textwidth=16cm]{geometry}
\usepackage{pdfpages}

\usepackage{xcolor}
\usepackage[colorlinks = true, linkcolor = blue, urlcolor = violet, citecolor = violet]{hyperref}

\usepackage{amsmath,amsfonts,amsthm, amssymb}
\usepackage{bbm,bm}
\usepackage{braket}
\usepackage{enumitem}

\theoremstyle{definition}
	\newtheorem{definition}{Definition}[section]
	\newtheorem{remark}[definition]{Remark}
	
	\newtheorem*{example*}{Example}
	\newtheorem*{mt}{Main Theorem}

\theoremstyle{plain}
	
	\newtheorem{lemma}[definition]{Lemma}
	\newtheorem*{lemma*}{Lemma}
	\newtheorem{proposition}[definition]{Proposition}
	\newtheorem{theorem}[definition]{Theorem}
	\newtheorem*{theorem*}{Theorem}
	
	\newtheorem{corollary}[definition]{Corollary}

\makeatletter
\newcommand{\specialcell}[1]{\ifmeasuring@#1\else\omit$\displaystyle#1$\ignorespaces\fi}
\makeatother

\newcommand{\numset}[1]{\mathbf{#1}}
	
	\newcommand{\rr}{\numset{R}}
	
	\newcommand{\zz}{\numset{Z}}
	\newcommand{\nn}{\numset{N}}
	\newcommand{\R}{\rr}
	\newcommand{\N}{\nn}

	\newcommand{\one}{\mathbf{1}}
	\newcommand{\e}{\mathrm{e}}
		\newcommand{\Exp}[1]{\e^{#1}}
	
	\makeatletter
	\providecommand*{\diff}%
		{\@ifnextchar^{\DIfF}{\DIfF^{}}}
	\def\DIfF^#1{%
		\mathop{\mathrm{\mathstrut d}}%
			\nolimits^{#1}\gobblespace}
	\def\gobblespace{%
		\futurelet\diffarg\opspace}
	\def\opspace{%
		\let\DiffSpace\!%
		\ifx\diffarg(%
			\let\DiffSpace\relax
		\else
			\ifx\diffarg[%
				\let\DiffSpace\relax
			\else
				\ifx\diffarg\{%
					\let\DiffSpace\relax
				\fi\fi\fi\DiffSpace}

	\renewcommand{\d}{\diff}
	\newcommand{\dd}{\diff}

	\newcommand{\pp}{\mathbf{P}}
	\newcommand{\ee}{\mathbf{E}}

\newcommand{\invar}{^\textnormal{inv}}

\newcommand{\simcompact}{I}
\newcommand{\simball}{J}
\newcommand{\simdiag}{K}
\newcommand{\stopping}{\mathcal{T}}

\newcommand{\XXXX}{\mathcal{X}}
\newcommand{\YYYY}{\mathcal{Y}}
\newcommand{\ZZZZ}{\mathcal{U}}
\newcommand{\xxxx}{{x}}
\newcommand{\yyyy}{{y}}
\newcommand{\zzzz}{{u}}
\newcommand{\FF}{\mathcal{F}}
\newcommand{\IP}{\mathbb{P}}
\newcommand{\PP}{\mathcal{P}}
\newcommand{\PPPP}{\mathfrak{P}}
\newcommand{\E}{\mathbb{E}}
\newcommand{\KK}{\mathcal{K}}
\newcommand{\RR}{\mathcal{R}}

\newcommand{\rectm}{\Sigma}

\newcommand{\bs}{\bold s}
\newcommand{\bxi}{\boldsymbol{\xi}}
\newcommand{\lag}{\langle}
\newcommand{\rag}{\rangle}
\newcommand{\T}{{\mathbb T}}
\newcommand{\HH}{{H}}
\newcommand{\HHH}{\mathcal{H}}

\newcommand{\VV}{{\mathcal V}}

\newcommand{\de}{\delta}
\newcommand{\La}{\Lambda}
\newcommand{\la}{\lambda}

	\newcommand{\cpoiss}{Y}
	\newcommand{\ASc}{\zeta}
	\newcommand{\AScaux}{\xi}
	\newcommand{\ASconstc}{\psi}
	\newcommand{\ASconstcaux}{\varphi}

\begin{document}

\title{Exponential mixing under controllability conditions for \textsc{sde}s driven by a degenerate Poisson noise}
\date{}
\author{
	Vahagn~Nersesyan%
	\textsuperscript{1,2,}%
	\and
	Renaud~Raqu\'epas%
	\textsuperscript{3,4}
}

\maketitle
\begin{center}
\small
\begin{tabular}{c c c c}
   1.  Laboratoire de math\'ematiques de Versailles
	 	&&& 3. McGill University \\
		CNRS, UVSQ, Universit\'e Paris-Saclay
	 	&&& Dept.\ of Mathematics and Statistics \\
	 	F-78\,035 Versailles
	 	&&& 1005--805 rue Sherbrooke~Ouest \\
	 France
	 	&&& Montr\'eal (Qu\'ebec)~H3A 0B9, Canada \\
		&&& \\
   2. Centre de recherches math\'ematiques  &&& 4. Univ. Grenoble Alpes \\
   CNRS, Universit\'e de Montr\'eal &&& CNRS, Institut Fourier \\
   CP~6129, Succursalle Centre-ville &&& F-38\,000 Grenoble \\
  	Montr{\'e}al (Qu{\'e}bec)~H3C 3J7, Canada &&& France\\
\end{tabular}
\end{center}

\begin{abstract}
	We prove existence and uniqueness of the invariant measure and  exponential mixing in the total-variation norm for a class of stochastic differential equations driven by degenerate compound Poisson processes. In addition to mild assumptions on the distribution of the jumps for the driving process, the hypotheses for our main result are that the corresponding control system is dissipative, approximately controllable and solidly controllable. The solid controllability assumption is weaker than the well-known parabolic H\"ormander condition and is only required from a single point to which the system is approximately controllable.
 Our analysis applies to Galerkin projections of stochastically forced parabolic partial differential equations with asymptotically polynomial nonlinearities and to networks of quasi-harmonic oscillators connected to different Poissonian baths.
	 \begin{description}
		 \item[Key words:] stochastic differential equations, Poisson noise, exponential mixing,
coupling, controllability, H\"ormander condition
		\item[MSC2010:] 60H10, 37A25, 93B05
	\end{description}
\end{abstract}

\setcounter{tocdepth}{1}
\tableofcontents


\section{Introduction}

Motivated by applications to thermally driven harmonic networks and to Galerkin approximations of partial differential equations (\textsc{pde}s) randomly forced by degenerate noise,
we consider a stochastic differential equation (\textsc{sde}) of the form
\begin{equation}\label{0.1}
	\dd X_t= f(X_t) \dd t+	B\dd \cpoiss_t,
\end{equation}
where $f:\R^d\to \R^d$ is a   smooth     vector field, $B:\R^n\to  \R^d$ is a linear map, and~$(\cpoiss_t)_{t \geq 0}$ is an~$n$-dimensional \textit{compound Poisson process} of the form
\begin{equation}\label{0.2}
	\cpoiss_t=\sum_{k=1}^\infty  \eta_k \one_{[\tau_k, \infty)}(t).
\end{equation}
{Throughout the paper, the jump displacements~$\{\eta_k\}_{k\in\nn}$ are independent and identically distributed random variables {with law~$\ell$} and the waiting times separating the jumps, defined as~$t_1=\tau_1$ and~$t_k=\tau_k-\tau_{k-1}$ for~$k \ge2$, form a sequence~$\{t_k\}_{k\in\nn}$ of independent exponentially distributed random variables with common rate parameter~$\la>0$.}
Moreover, the sequences~$\{\eta_k\}_{k\in\nn}$ and~$\{t_k\}_{k\in\nn}$ are independent from one another. {We are interested in the \emph{noise-degenerate} case, that is when $\operatorname{rank}(B) < d$.}


The aim of this paper is to establish exponential mixing for the \textsc{sde}~\eqref{0.1} under some mild dissipativity and controllability conditions. The precise hypotheses are the following.
\begin{itemize}
	\item[(C1)] There are  numbers $\alpha>0$ and $\beta>0$   such that
	\begin{equation}\label{0.3}
	 	\braket{f(y), y }\le -\alpha \|y\|^2 +\beta
	\end{equation}
	{for all $y\in \R^d$, where $\braket{{\,\cdot\,},{\cdot\,}}$ and $\|{\,\cdot\,}\|$ are a scalar product and the associated norm  in $\R^d$.}
\end{itemize}
{Combined with the regularity of~$f$ and the fact that $\sum_{k=1}^\infty t_k=+\infty$ with probability 1, it ensures the global well-posedness of the \textsc{sde}~\eqref{0.1}. It also strongly suggests the norm squared as a candidate Lyapunov function.}
The other two conditions are related to the controllability of the system{: we ask that there exists a point~$\hat{x} \in \rr^d$ such that the system is \emph{both} approximately controllable to~$\hat{x}$ \emph{and} solidly controllable form~$\hat{x}$. To formulate these conditions more precisely}, we introduce  the following (deterministic) mapping. For $T>0$ a given time,
  \begin{equation}\label{0.4}
	\begin{split}
		S_T: \R^d\times C([0,T]; \R^n) &\to \R^d, \\
			(x,\zeta) &\mapsto y_T,
	\end{split}
  \end{equation}
 where $(y_t)_{t \in [0,T]}$ is the solution of the controlled problem
\begin{equation}\label{0.5}
	\begin{cases}
		\dot y_t = f(y_t) + B \zeta_t, \\
		y_0 = x.
	\end{cases}
\end{equation}
Accordingly, we will refer to the first argument of~$S_T({\,\cdot\,},{\cdot\,})$ as an  initial condition and to the second one  as  a control.

\begin{itemize}
	\item[(C2)] The system is \textit{approximately controllable to} $\hat x \in \R^d$: for any number $\epsilon>0$ and any radius $R>0$, we can find a time~$T>0$ such that for any initial point $x\in \R^d$ with~$\|x\|\leq R$, there exists a control~$\zeta\in C([0,T]; \R^n)$ verifying
	\begin{equation}
		\|S_T(x,\zeta)-\hat x\|<\epsilon.\label{0.6}
	\end{equation}

	\item[(C3)] The system is \textit{solidly controllable from} $\hat x$: there is a  number $\epsilon_0>0$, a time $T_0>0$, a compact set $\KK$ in~$C([0,T_0]; \R^n)$ and a non-degenerate ball
	$G$ in $\R^n$  such that, for any continuous function $\Phi:\KK\to \R^d$ satisfying the relation
	$$
	\sup_{\zeta\in \KK} \|\Phi(\zeta)-S_{T_0}(\hat x,\zeta)\|\le \epsilon_0,
	$$we have $G\subset \Phi(\KK)$.
\end{itemize}
Condition~\textnormal{(C2)} is a well-known controllability property, and~\textnormal{(C3)} is an   accessibility property that is weaker than the weak  H\"ormander condition at the point $\hat x$ (see Section~\ref{A:B} for a discussion).

We denote by $(X_t, \IP_x)$ the Markov family associated with the \textsc{sde}~\eqref{0.1}
parametrised by the time~$t \geq 0$ and the initial condition $x\in \rr^d$, by~$P_t(x,{\cdot\,})$ the corresponding transition function, and by
$\PPPP_t $ and~$\PPPP_t^* $   the Markov semigroups
\[
\PPPP_t g(x)=\int_{\rr^d}g(y)\,P_t(x,\d
y) \qquad\text{ and }\qquad\PPPP^*_t\mu(\Gamma)=\int_{\rr^d}
P_t(y,\Gamma)\,\mu(\d y),
\]
where $g\in L^\infty(\R^d)$ and $\mu\in {\mathcal{P}}(\rr^d).$ Recall that a measure~$\mu\invar\in {\mathcal{P}}(\R^d)$ is said to be \emph{invariant}   if~$\PPPP^*_t\mu\invar=\mu\invar$ for all $t\ge0$.
\begin{mt}
	Assume that  Conditions~\textnormal{(C1)--(C3)} are satisfied and that
	the law of~$\eta_k$ has finite variance and~possesses a continuous positive density  with respect to the Lebesgue measure on~$\rr^n$.
	Then, the semigroup $(\PPPP^*_t)_{t \geq 0}$  	admits a unique invariant measure~$\mu\invar \in \mathcal{P}(\rr^d)$. Moreover, there exist constants $C > 0$ and~$c > 0$ such that
	\begin{equation}\label{0.7A}
		\|\PPPP^*_t\mu - \mu\invar \|_{\textnormal{var}} \leq C\, \Exp{-c t} \left( 1 + \int_{\rr^d} \|x\| \,\mu(\d x)\right)
	\end{equation}
	for any $\mu\in\mathcal{P}(\rr^d)$ and~$t \geq 0$.
\end{mt}
In the literature, the problem of ergodicity for \textsc{sde}s driven by a degenerate noise is mostly considered when the perturbation is a Brownian motion, {the system admits a Lyapunov function}, and the H\"ormander condition is satisfied at all the points of the state space. Under these   assumptions, the transition function of   the underlying Markov process has a   smooth density with respect to Lebesgue measure which is almost surely positive.
This implies that the process is \textit{strong Feller} and \textit{irreducible}, so it has a unique invariant measure by Doob's theorem (see Theorem~4.2.1 in~\cite{dPZa} and~\cite{MT1993,H-2012} for related results).

{Even with the assumption that the noise is Gaussian,} there are only few papers that consider the problem of ergodicity for an \textsc{sde} without the H\"ormander condition being satisfied everywhere. In~\cite{AK87}, the uniqueness property for invariant measures is proved for degenerate diffusions, under the assumption that the H\"ormander condition holds at one point and that the process is irreducible. The proof relies heavily on the Gaussian nature of the noise.
In the paper~\cite{Sh17}, an approach based on controllability and a coupling argument is given for a study of dynamical systems on compact metric spaces subject to a more general degenerate noise: under the controllability assumptions~\textnormal{(C2)} and~\textnormal{(C3)} and a \textit{decomposability} assumption on the noise, exponential mixing in the total-variation metric is established. This approach can be carried to problems on a non-compact space, provided a dissipativity of the type of~\textnormal{(C1)} holds; see~\cite{Ra18} for a study of networks of quasi-harmonic oscillators. The class of decomposable noises includes\,---\,but is not limited to\,---\,Gaussian measures.

The present paper falls under the continuity of the study carried out in these references. The main difficulty in our case comes from the fact that the Poisson noise we consider, in addition to being degenerate, does not have a decomposability structure; also see~\cite{Ne08}, where polynomial mixing is proved for the complex Ginzburg--Landau equation driven by a non-degenerate compound Poisson process. Yet, the methods we use still stem from a control and coupling approach, which we outline in the following paragraphs; also see the beginning of Section~\ref{sec:coupling}. Indeed, the combination of coupling and controllability arguments has the advantage of yielding rather simple proofs of otherwise very technical results and also accommodates a wide variety of (non-Gaussian) noises for which other methods fail.

We hope that treating a relatively tractable problem in an essentially self-contained way will help interested readers in making their way to understanding technically more difficult problems for which methods of the same flavour are used.

\bigskip

For a discrete-time Markov family on a compact state space~$\mathcal X$, existence of an  invariant measure can be obtained from a Bogolyubov--Krylov argument and it is typical to derive uniqueness and mixing from a uniform upper bound on the total-variation distance between the transition functions from different points. One way to prove uniqueness using such a uniform \textit{squeezing} estimate is through a so-called Doeblin coupling argument, where one constructs a Markov family on~$\mathcal X \times \mathcal X$ whose projections to each copy of~$\mathcal X$ have the same distribution as the original Markov family, and with the property that {it hits the diagonal $\{(x,x) : x \in \mathcal{X}\}$ soon enough, often enough}. We refer the interested reader to the paper~\cite{Gr75} and to Chapter~3 of the monograph~\cite{KuSh} for an introduction to these ideas, which go back to Doeblin, Harris, and Vaserstein.

When the state space~$\mathcal X$ is not compact, existence of an  invariant measure requires additional arguments and one can rarely hope to prove squeezing estimates which hold uniformly on the whole state space. The Bogolyubov--Krylov argument for existence can be adapted provided that one has a suitable Lyapunov structure. As for uniqueness and mixing, the coupling argument will go through with a squeezing estimate which only holds for points in a small ball, provided that one can obtain good enough estimates on the hitting time  of that ball. Over the past years, it has become evident that control theory provides a good framework for formulating conditions that are sufficient for this endeavour when the noise is degenerate.

\paragraph*{Acknowledgements}
This research was supported by the {Agence Nationale de la Recherche} through  the grant NONSTOPS (ANR-17-CE40-0006-01, ANR-17-CE40-0006-02, ANR-17-CE40-0006-03). VN was supported by the CNRS PICS \textit{Fluctuation theorems in stochastic systems}. The research of RR was supported by the National Science and Engineering Research Council (NSERC) of Canada. Both authors would like to thank No\'{e} Cuneo, Vojkan Jak\v{s}i\'{c}, Claude-Alain Pillet and Armen Shirikyan for discussions and comments on this manuscript.

\subsubsection*{Notation}
For $(\mathcal{X},d)$ a  Polish space, we shall use the following notation throughout the paper:
\begin{itemize}
	\item $B_\mathcal{X}(x,\epsilon)$ for the closed ball in~$\mathcal{X}$ of radius~$\epsilon$ centred at~$x$ (we shall simply write $B(x,\epsilon)$ in the special case $\mathcal{X}=\rr^d$);
	\item $\mathcal{B}(\mathcal{X})$ for its Borel $\sigma$-algebra;
	\item $L^\infty(\mathcal{X})$ for the space of all bounded Borel-measurable functions $g:\mathcal{X}\to \rr$, endowed with the norm $\|g\|_\infty=\sup_{y\in \mathcal{X}} |g(y)|$;
	\item $\PP(\mathcal{X})$ for the set of Borel probability measures on $\mathcal{X}$, endowed with the total variation norm: for $\mu_1, \mu_2\in \PP(\mathcal{X})$,
	\begin{align*}
	  \|\mu_1-\mu_2\|_{\textnormal{var}}&:=\frac12\sup_{\|g\|_\infty\le 1}|\lag g,\mu_1\rag-\lag g,\mu_2\rag|\\&=\sup_{\Gamma\in \mathcal{B}(\mathcal{X})} | \mu_1(\Gamma)- \mu_2(\Gamma)|,
	\end{align*}    where    $\lag g,\mu\rag=\int_{\mathcal{X}}g(y)\,\mu(\d y)$ for $g\in L^\infty(\mathcal{X})$ and $\mu \in {\mathcal{P}}(\mathcal{X})$.
\end{itemize}

Let $(\mathcal{Y},d')$ be another Polish space. The image of a measure $\mu\in\PP(\mathcal{X})$ under a Borel-measurable mapping $F:\mathcal{X}\to \mathcal{Y}$ is  denoted by $F_*\mu \in \PP(\mathcal{Y})$.

On any space, $\one_\Gamma$ stands for the indicator function of the set $\Gamma$. 

We use~$\zz$ for the set of integers and~$\N$ for the set  of natural numbers (without~0). For any $m\in \N$, we set
\begin{equation}\label{0.9}
 \N_m :=\{n \cdot  m: n\in \N\}\quad \text{ and } \quad \N_m^0:= \N_m\cup\{0\}.
\end{equation}

\noindent
We use $a\vee b$ [resp.~$a \wedge b$] for the maximum [resp.~minimum] of the numbers $a,b\in \rr$.

\section{Preliminaries and existence of an invariant measure}
\label{sec:prelim}

The \textsc{sde}~\eqref{0.1} has a unique   c\`adl\`ag solution satisfying the  initial condition $X_0=x.$
It is given by
\begin{equation}\label{eq:disc-to-cont}
 X_t =
	\begin{cases}
   	S_{t-\tau_k}(X_{\tau_k})       & \quad \text{if } t\in [\tau_k,\tau_{k+1}), \\
   	S_{t_{k+1}}(X_{\tau_{k}})+B\eta_{k+1}   & \quad \text{if } t=\tau_{k+1},
 	\end{cases}
\end{equation}
where $\tau_0=0$ and $S_t(x)=S_t(x,0)$ is the solution of the  undriven equation. Relation~\eqref{eq:disc-to-cont} will allow us to reduce the study of the ergodicity of the full process~$(X_t)_{t\geq0}$ to that of the embedded process~$(X_{\tau_k})_{k\in\nn}$ obtained by considering its values at jump times $\tau_k$. The strong Markov property implies that the latter is a Markov process with respect to the filtration generated by the random variables~$\{t_j, \eta_j\}_{j=1}^k$.
We denote by $\hat P_k$ the corresponding transition function: for~$x\in \rr^d$ and~$\Gamma\in \mathcal{B}(\rr^d)$,
\begin{equation}
\label{eq:def-P-hat}
	\hat P_k(x,\Gamma) := \IP_x\left\{X_{\tau_{k}}\in \Gamma\right\}.
\end{equation}
The key consequences of the dissipativity Condition~\textnormal{(C1)} are the moment estimates of the following lemma. They imply, in particular, existence of a  suitable Lyapunov structure given by   the norm squared.
\begin{lemma}\label{lem:Lyap}
	Under Condition~\textnormal{(C1)}, we have the following bounds:
	\begin{itemize}
		\item[(i)] for any $\epsilon > 0$, there exists a constant $C_\epsilon > 0$ such that
		\begin{equation}\label{eq:det-bound-emb}
			\|X_{\tau_k}\|^2 \leq (1+\epsilon)^k \Exp{-2\alpha\tau_k} \|X_0\|^2 + C_\epsilon \sum_{j=1}^k \Exp{-2\alpha(\tau_k-\tau_j)}(1+\epsilon)^{k-j}(1+\|\eta_j\|^2)
		\end{equation}
		for all~$x \in \rr^d$ and~$k \in \nn$;
		\item[(ii)] there are numbers $ \gamma\in (0,1)$ and $C>0$
		such that
		\begin{align}
			\E_x \|X_{\tau_k}\|^2 &\leq \gamma^k \|x\|^2+ C(1+\Lambda), \label{eq:moment-emb}\\
			\E_x \|X_t\|^2 &\leq (1-\gamma)^{-1}\|x\|^2 + C(1+\Lambda)\label{eq:moment-cont}
		\end{align}
		for all~$x \in \rr^d, k\in \nn,$ and~$t \geq 0$, where  $ \Lambda := \E \|\eta_1\|^2$ and $\E_x$ is the expectation with respect~to~$\IP_x$.
	\end{itemize}
\end{lemma}

\begin{proof}
	First note that Condition~\textnormal{(C1)} implies the following estimate for the solution to the undriven equation:	\begin{equation}\label{eq:Gr}
		\|S_t(x)\|^2 \leq \Exp{-2 \alpha t }\|x\|^2 + \beta\alpha^{-1}
	\end{equation}for all $x \in \rr^d$ and~$t \geq 0$.
		Let $\epsilon > 0$ be arbitrary. Combining~\eqref{eq:disc-to-cont} and~\eqref{eq:Gr}, we find a positive constant~$C_\epsilon$ such that
		\[
			\|X_{\tau_k}\|^2 \leq (1+\epsilon)\Exp{-2\alpha t_k} \|X_{\tau_{k-1}}\|^2 + C_\epsilon (1+\|\eta_k \|^2).
		\]
		Iterating this inequality, we get~\eqref{eq:det-bound-emb}.
	 Taking expectation in~\eqref{eq:det-bound-emb} and using the independence of the sequences~$\{\eta_k\}$ and~$\{\tau_k\}$, we obtain
		\[
			\E_x \|X_{\tau_k}\|^2 \leq (1+\epsilon)^k \left(\frac{\lambda}{\lambda +2\alpha}\right)^k \|x\|^2 + C_\epsilon \sum_{j=1}^k \left(\frac{\lambda}{\lambda +2\alpha}\right)^{k-j} (1+\epsilon)^{k-j}(1+\La).
		\]
		Choosing~$\epsilon > 0$ so small that $\gamma := (1+\epsilon)\tfrac{\lambda}{\lambda + 2\alpha} \in (0,1)$ yields~\eqref{eq:moment-emb}. To prove~\eqref{eq:moment-cont}, we introduce the random variable
		\[
			\mathcal{N}_t := \max\{k \geq 0 : \tau_k \leq t\}
		\]
		and use~\eqref{eq:Gr}:
		\begin{equation}\label{eq:cont-time-decomp}
			\E_x\|X_t\|^2 \leq \E_x \|X_{\tau_{\mathcal{N}_t}}\|^2 + \beta\alpha^{-1} =\sum_{k=0}^\infty \E_x \left(\one_{\{\mathcal{N}_t = k\}}\|X_{\tau_k}\|^2\right)+  \beta\alpha^{-1} .
		\end{equation}
		 Inequality~\eqref{eq:det-bound-emb} and the independence of  $\{\eta_k\}$ and~$\{\tau_k\}$ imply
		\begin{equation}\label{eq:bound-N-is-k}
				\E_x \left(\one_{\{\mathcal{N}_t = k\}}\|X_{\tau_k}\|^2\right) \leq \gamma^k\|x\|^2 + C_\epsilon(1+\Lambda)\sum_{j=1}^k(1+\epsilon)^{k-j} \E \left(\one_{\{\mathcal{N}_t = k\}} \Exp{-2\alpha(\tau_k-\tau_j)}\right)
		\end{equation}
		and
		$$
			\sum_{k=1}^\infty\sum_{j=1}^k (1+\epsilon)^{k-j} \E \left(\one_{\{\mathcal{N}_t = k\}} \Exp{-2\alpha(\tau_k-\tau_j)}\right) = \sum_{k=0}^\infty (1+\epsilon)^k \E\left(\Exp{-2\alpha\tau_k}\right) = \sum_{k=0}^\infty (1+\epsilon)^k \left(\frac{\lambda}{\lambda +2\alpha}\right)^k,
		$$
		which is finite by our choice of~$\epsilon$. Combining this with~\eqref{eq:cont-time-decomp} and~\eqref{eq:bound-N-is-k}, we get~\eqref{eq:moment-cont} and complete  the proof~of~the lemma.
\end{proof}

As mentioned  in the introduction, the dissipativity Condition~\textnormal{(C1)} guarantees the existence of an  invariant measure. Indeed, the last lemma, combined with a Bogolyubov--Krylov argument and Fatou's lemma yields the following   result. We refer the  reader   to~\cite[\S{2.5}]{KuSh} for more details.
\begin{lemma}\label{lem:existence}
   Under Condition~\textnormal{(C1)}, the semigroup $(\PPPP^*_t)_{t \geq 0}$ admits at least one invariant measure~$\mu\invar\in \PP(\R^d)$. Moreover, any invariant measure~$\mu\invar\in \PP(\R^d)$   has a finite second~moment, that is
	 \begin{equation}\label{1.8}
		\int_{\R^d}\|y\|^2 \,\mu\invar(\dd y)< \infty.
	\end{equation}
\end{lemma}
We now turn to an important consequence of the solid controllability Condition~\textnormal{(C3)}. The main ideas in its proof are borrowed from {\cite[\S{1}]{Sh17}  (also see the earlier works \cite[\S{2}]{AKSS07} and \cite[Ch.\,3]{KuSh})}. Such results are sometimes referred to as squeezing estimates, a concept to which we have referred in the introduction. This lemma  is used to prove a key property of the coupling constructed in the next section.

We consider the family of maps $F_k : \rr^d \times (\rr_+)^{\nn} \times (\rr^n)^\nn \to \rr^d$ defined by
\begin{equation}\label{E:fk}
	\begin{cases}
		F_0(x,\bs,  \bxi) = x, \\
			F_k(x,\bs, \bxi) = S_{s_k}(F_{k-1}(x,\bs, \bxi)) + B\xi_k
	\end{cases}
\end{equation}
for $k \in \nn$, $x\in \rr^d$, $\bs=(s_j)_{j\in\nn}\in (\rr_+)^{\nn}$, and $\bxi=(\xi_j)_{j\in\nn}\in (\rr^n)^\nn$; see Figure~\ref{fig:jump-flow}.
Because $F_k$ does not depend on $\{s_j,\xi_j\}_{j\ge k+1}$, {i.e.\ the times and displacements for kicks that happen later than the $k$-th kick,} we will often consider the domain of~$F_k$ to be~$\rr^d \times (\rr_+)^{m} \times (\rr^n)^{m}$ for some natural number~$m \geq k$.

\begin{figure}
\centering
	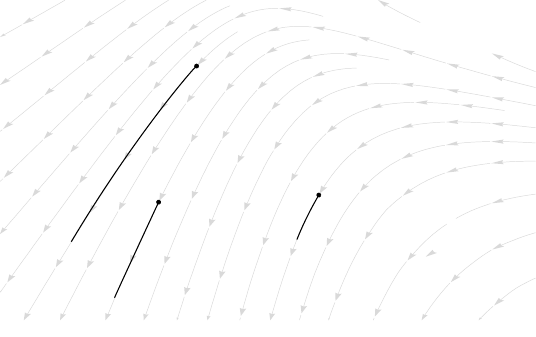
	\caption{{The map $F_k$ takes as an input a point~$x$, a sequence~$\bs$ of times and a sequence~$\bxi$ of displacement vectors and outputs the final position of a test particle which starts at~$x$, follows the integral curves of~$f$ for a time~$s_1$, is immediately displaced by~$\xi_1$, follows the integral curves of~$f$ for a time~$s_2$, is immediately displaced by~$\xi_2$, and so on until it is finally displaced by~$\xi_k$. We have sketched this for~$k=4$.}}
	\label{fig:jump-flow}
\end{figure}

\begin{lemma}\label{lem:minor-tot-var}
	{Suppose that $\hat{x}$ is as in Condition~\textnormal{(C3)}. Then,}  there exist numbers $m \in \nn$,   $r > 0$, and $p \in (0,1)$ and a non-degenerate ball\,\footnote{Here $[0,T_0]^{m}$ is endowed with the metric inherited from $\rr^{m}$.} $\rectm$ in $[0,T_0]^{m}$ such that
	\begin{equation}\label{E:18}
		\left\| F_m(x,\bs,{\cdot\,})_* (\ell^m ) -  F_m(x',\bs,{\cdot\,})_* (\ell^m) \right\|_\textnormal{var} \leq p
	\end{equation}
	for all $\bs \in \rectm$ and~$x,x'\in B(\hat x, r)$, where $F_m(x,\bs,{\cdot\,})_* (\ell^m )$ is the image of $\ell^m$ (the $m$-fold product of the law~$\ell$ with itself) under the mapping~$F_m(x,\bs,{\cdot\,}):(\rr^n)^m\to \rr^d$.
\end{lemma}

\begin{proof}
	Let us fix $\epsilon_0$, $\KK$, and $G$ as in Condition~\textnormal{(C3)}. To simplify the presentation, we assume that   $T_0=1$.
	For any
  $m \in \nn$ and $\zeta\in C([0,1];\rr^n)$, let $\iota_m(\zeta):[0,1]\to \rr^n$ be the step function
	\[
		\iota_m(\zeta)=  \sum_{j=0}^{m-1}  \one_{\left[\frac {j}m,\frac{j+1}{m}\right)} \int_0^{\frac jm}\zeta(s)\d s,
	\] and    let
	$\KK_m$ be the set $\iota_m (\KK)$. {If~$\zeta$ is a continuous function which allows the system to be controlled from~$\hat x$ to some target in time~$1$, then~$\iota_m(\zeta)$ is a discretization in time of the antiderivative of~$\zeta$ and we expect that feeding its jump discontinuities to~$F_m$ would result in a final position which is close to the target if~$m$ is large enough. With this in mind, we often identify the function~$\iota_m(\zeta)$ with the $m$-tuple of vectors in~$\rr^n$ consisting of its jumps at the times~$\tfrac 1m, \tfrac 2m, \dotsc, \tfrac mm$.}

	We proceed in three steps. We first show that Condition~\textnormal{(C3)} implies that the set $F_m(\hat x, \hat \bs, \KK_m)$ contains a ball in~$\rr^d$.
	Then, combining this with Sard's theorem and some properties of images of measures under regular mappings, we show   a uniform lower bound on $ F_m(x,\bs,{\cdot\,})_* (\ell^m)$ for~$(x,\bs)$ close enough to~$(\hat x,\hat \bs)$ where $ \hat \bs := (\tfrac{1}{m}, \dotsc, \tfrac{1}{m}) \in [0,1]^{m}$.
	Finally, from this uniform lower bound  we derive   the desired estimate in total variation.

	\smallskip

	\noindent\textit{Step 1: Solid controllability.}
		Let $S_T$ be the mapping defined by~\eqref{0.4}.
		By the compactness of~$\KK$, for any $\epsilon>0$,  there exists $m_0(\epsilon) \in \nn$ such that
		$$
		   \sup_{\zeta\in \KK} \Big\|\iota_m\zeta-\int_0^\cdot \zeta(s)\d s\Big\|_{L^\infty([0,1],\rr^n)}\le \epsilon
		$$
		whenever $m \geq m_0(\epsilon)$. Hence, taking $m \geq m_0(\epsilon)$ for sufficiently small~$\epsilon$, we have
		$$
			\sup_{\zeta \in \KK} \| F_m(\hat x, \hat \bs, \iota_m\zeta) - S_1(\hat x,\zeta)\| \leq  \epsilon_0,
		$$
		{where we use the aforementioned identification of functions in~$\KK_m$ with $m$-tuples of displacement vectors in~$\rr^n$}.
		Using the continuity of $F_m(\hat x, \hat \bs, \iota_m \cdot): \KK \to \rr^d $ and Condition~\textnormal{(C3)}, we conclude that $F_m(\hat x, \hat \bs, \KK_m)$ contains a ball in~$\rr^d$.
		Until the end of the proof, we fix~$m \geq m_0(\epsilon)$ for such a small~$\epsilon$.

	\smallskip

	\noindent\textit{Step 2: Uniform lower bound.}
		We want to apply Lemma~\ref{lem:image-meas}  with  $\XXXX ={B(\hat x,1)} \times [0,1]^{m}$, $\YYYY = \rr^d$, and $\ZZZZ = (\rr^{n})^{m}$ and   the map $F_m : \XXXX \times \ZZZZ \to \YYYY$ as before.
		As $F_m(\hat x, \hat \bs, \KK_m)$ contains a ball in~$\rr^d$, Sard's theorem yields the existence of a point $\hat\zzzz \in \KK_m \subset \ZZZZ$ in which the derivative $D_{\bxi} F_m(\hat x, \hat \bs,{\cdot\,})$ has full rank.
		Hence, by Lemma~\ref{lem:image-meas}, there exists a continuous function $\psi : \XXXX \times \YYYY \to \rr_+$ and a radius~$r_m > 0$ such that
		\[
			\psi\left((\hat x,\hat \bs), F_m(\hat x, \hat \bs, \hat\zzzz)\right) > 0
		\]
		and
		\[
		 	\left(   F_m(x, \bs,{\cdot\,})_*(\ell^m)\right)(\d y)  \geq \psi\left((x,\bs),\yyyy\right) \d\yyyy
		\]
		(as measures, with $y$ ranging over~$\rr^d$) whenever $ x\in B(\hat x, r_m)$ and $\bs\in B_{\rr^m}(\hat \bs, r_m)$.

	\smallskip

	\noindent\textit{Step 3: Estimate in total variation.}
		Shrinking $r_m$ if necessary, Step 2 yields positive numbers~$\epsilon_{m,1}$ and~$\epsilon_{m,2}$ and a non-degenerate ball~$\rectm \subset  [0,1]^{m}$ such that
		\[
			F_m(x, \bs,{\cdot\,})_* (\ell^m) \wedge F_m(x', \bs,{\cdot\,})_* (\ell^m) \geq \epsilon_{m,1} \operatorname{Vol}_{\rr^d}\left({\,\cdot\,} \cap B(F_m(\hat x, \hat \bs, \hat\zzzz),\epsilon_{m,2})\right)
		\]
		whenever $x,x' \in B(\hat x, r_m)$ and $\bs \in \rectm$.
		Therefore,
		\[
			\|F_m(x,\bs,{\cdot\,})_* (\ell^m) - F_m(x',\bs,{\cdot\,})_* (\ell^m)\|_\textnormal{var} \leq 1 - \epsilon_{m,1} \epsilon_{m,2}^d \frac{\pi^{\frac d2}}{\Gamma\left(\frac d2 + 1\right)}=:p_m
		\]
		whenever  $x,x' \in B(\hat x, r_m)$ and $\bs \in \rectm$. This proves~\eqref{E:18} with $r=r_m$ and $p=p_m$.
 \end{proof}

\section{Coupling argument and exponential mixing}\label{sec:coupling}

In this section, we shall always assume that Conditions~\textnormal{(C1)}--\textnormal{(C3)} are satisfied. The Main Theorem is established by using the
  \textit{coupling method}, which consists in proving uniqueness and convergence to an invariant measure for a Markov family by using the inequality
$$
	\|P_t(x,{\cdot\,}) - P_t(x',{\cdot\,}) \|_\textnormal{var} \leq \IP\{ \stopping> t\},
$$
where $\stopping$ is a random time given by
\begin{equation}\label{E:sigma}
\stopping:=\inf\left\{ s\ge0 : Z_{u} = Z'_{u} \text{ for all } u \geq s\right\}
\end{equation}
and $(Z_t,Z'_t)_{t\geq 0}$ is any $(\rr^d\times\rr^d)$-valued random process defined  on {a space~$(\Omega,\mathcal{F},\IP_{(x,x')})$ with~$\IP_{(x,x')}(Z_t \in \Gamma) = P_t(x,\Gamma)$ and $\IP_{(x,x')}(Z'_t \in \Gamma) = P_t(x',\Gamma)$ for all $t \geq 0$ and all measurable~$\Gamma \subseteq \rr^d$.}
This inequality is of course most useful when the process~$(Z_t,Z'_t)_{t\geq 0}$, called a  \textit{coupling}, is constructed in a such a way that $\IP\{\stopping> t\}$ decays as fast as possible as $t \to \infty$, with a reasonable dependence on~$x$ and~$x'$. To do so, one usually uses at some point a general result of the type of Lemma~\ref{lem:m-c} on the existence of so-called \textit{maximal couplings} (see \cite[Chapter~3]{KuSh}).

{We first proceed to construct a coupling of two embedded discrete-time processes as introduced at the beginning of Section~\ref{sec:prelim},  but with different initial conditions: given~$x$ and~$x'$ in~$\rr^d$, we define a sequence~$(z_k,z'_k)_{k \in \nn}$ of $(\rr^d\times\rr^d)$-valued random variables on a probability space~$(\Omega,\mathcal{F},\IP_{(x,x')})$ with $\IP_{(x,x')}(z_k \in \Gamma) = \hat P_{k}(x,\Gamma)$ and $\IP_{(x,x')}(z'_k \in \Gamma) = \hat P_{k}(x',\Gamma)$ for all~$k \in \nn$ and measurable~$\Gamma \subseteq \rr^d$}.
{In this context, we call $(z_k)_{k \in \nn}$ [resp. $(z_k)_{k \in \nn}$] the first [resp.\ second] \emph{component} of the coupling $(z_k,z'_k)_{k \in \nn}$.}
The structure of the waiting times and the relation~\eqref{eq:disc-to-cont} then allow us to recover estimates for the original continuous-time process.
{The construction of this coupling is inductive and relies on the numbers~$m \in \nn$  and~$r > 0$ in Lemma~\ref{lem:minor-tot-var} and correlates the two components in a different way according to three cases}: for $j \in \nn_m^0$,
\begin{itemize}
	\item if $z_j = z'_j$, then $z_k = z'_k$ for all $k \in \nn$ with $k \geq j$;
	\item if $z_j$ and $z'_j$ are different but both in $B(\hat x, r)$, then the next~$m$ jumps are synchronous and{, given the times of these jumps,~$z_{j+m}$ and $z'_{j+m}$ are maximally coupled in the sense of} Lemma~\ref{lem:m-c};
	\item if $z_j$ and $z'_j$ are different and not both in $B(\hat x, r)$, then the next~$m$ jumps are synchronous, but the respective jump displacements are independent.
\end{itemize}
In essence, the worst-case scenario is when the initial conditions~$x$ and~$x'$ are different and very far from the origin, but the number
\begin{equation}\label{eq:def-simcompact}
	\simcompact := \min\{i \in \nn_m^0 : (z_i, z'_i) \in B(0,R) \times B(0,R)\}
\end{equation}
of jumps needed for both components to enter a large\footnote{The radius~$R$ of this compact set will be chosen to suitably fit the Lyapunov structure; cf.~Corollary~\ref{lem:R4}.} compact {set} around the origin is controlled by the Lyapunov structure inherited from~\textnormal{(C1)}. Then, the approximate controllability assumption~\textnormal{(C2)} {allows us to prove an estimate for an exponential moment of} the number
\begin{equation}\label{eq:def-simball}
	\simball := \min\{j \in \nn_m^0 : (z_j, z'_j) \in B(\hat x, r) \times B(\hat x,r)\}
\end{equation}
of jumps needed for both components to simultaneously enter~$B(\hat x, r)$. Finally, combining this with {the solid controllability assumption~\textnormal{(C3)},
we control the probability distribution of} the number
\begin{equation}\label{eq:def-simdiag}
\begin{split}
	\simdiag &:= \min\{k \in \nn_m^0 : z_k = z'_k\} \\
		&\ = \min\{k \in \nn_m^0 : z_\ell = z'_\ell \text{ for all } \ell \in \nn \text{ with } \ell \geq k\}
\end{split}
\end{equation}
of jumps after which the two components coincide.

{Alternatively, in a language which avoids the particularities of the coupling method, one could rephrase the above strategy by saying that combining (C2) and the consequence of (C3) expressed in Lemma~\ref{lem:minor-tot-var} gives a \emph{local Doeblin condition} in $B(0,R)$ which, when combined with the Lyapunov structured conferred by~(C1), yields exponential mixing by Meyn--Tweedie-type arguments~\cite{MeTw}.}

\subsection{Coupling for the embedded discrete-time process}\label{S:3.1}

In this section, {we construct a coupling~$(z_k,z_k')_{k\in\nn}$ for the embedded discrete-time process in such a way that the random time after which the two components coincide has an exponential moment which can we estimate in terms of the initial conditions} (see Proposition~\ref{prop:coincide}).

Let us fix the numbers   $m$, $r$, and $p$ as in Lemma~\ref{lem:minor-tot-var}.  The coupling is constructed    by blocks of~$m$~steps as follows.
Let $\XXXX = \rr^d \times (\rr_+)^m \times (\rr^n)^m$, $\YYYY = \rr^d$, and $\ZZZZ =\rr^d\times\rr^d\times (\rr_+)^m$.
Recall that   the functions $F_i : \XXXX \to \YYYY$ are defined by~\eqref{E:fk}
for $i = 1, \dotsc, m$. We consider 	two random probability measures
$\zzzz\in \ZZZZ\mapsto\mu(\zzzz,{\cdot\,}),\,\mu'(\zzzz,{\cdot\,})$   on~$\XXXX$ given by
\[
		\mu(\zzzz,{\cdot\,}) := \delta_{z} \times \delta_{\bs} \times \ell^m
	\qquad \text{and} \qquad
		\mu'(\zzzz,{\cdot\,}) := \delta_{z'} \times \delta_{\bs} \times \ell^m
\]
for $u=(z,z',\bs)\in \ZZZZ$, where $\delta_z$ is the Dirac measure at $z \in \rr^d$ and $\delta_{\bs}$ is the Dirac measure at $\bs \in (\rr_+)^m$.
{By Lemma~\ref{lem:m-c} applied to~$F_m$, there exist
a probability space $(\tilde\Omega,\tilde\FF,\tilde\IP)$ and measurable mappings $\xi, \xi':\ZZZZ\times \tilde\Omega\to \XXXX$ such that
\begin{gather}
		\xi(\zzzz,{\cdot\,})_* \tilde{\IP} = \delta_{z} \times \delta_{\bs} \times \ell^m, \qquad
		\xi'(\zzzz,{\cdot\,})_* \tilde{\IP} = \delta_{z'} \times \delta_{\bs} \times \ell^m,\nonumber
\intertext{and}
		\tilde{\IP}\big\{\tilde\omega: F_m(\xi(\zzzz,\tilde\omega)) \neq F_m(\xi'(\zzzz,\tilde\omega))\big\}
			= \left\|  F_m(z,\bs,{\cdot\,})_*(\ell^m) -    F_m(z',\bs,{\cdot\,})_*(\ell^m)\right\|_{\textnormal{var}} \label{eq:coupl-tot-var}
\end{gather}
for each $u=(z,z',\bs)\in \ZZZZ$.}
Replacing~$\tilde\Omega$ with a bigger space (still referred to as~$\tilde\Omega$) if necessary, we may find a third  measurable mapping $\xi'':\ZZZZ\times \tilde\Omega\to \XXXX$   with the same distribution  as~$\xi'$, but independent from~$\xi$.%
\footnote{{For example, one can take as a new~$(\tilde\Omega,\tilde\FF,\tilde\IP)$ the product of the old~$(\tilde\Omega,\tilde\FF,\tilde\IP)$ with itself and set $\xi_\textnormal{new}(u,\tilde\omega_1,\tilde\omega_2) = \xi_\textnormal{old}(u,\tilde\omega_1)$, $\xi'_\textnormal{new}(u,\tilde\omega_1,\tilde\omega_2) = \xi'_\textnormal{old}(u,\tilde\omega_1)$ and $\xi''_\textnormal{new}(u,\tilde\omega_1,\tilde\omega_2) = \xi'_\textnormal{old}(u,\tilde\omega_2)$ where $(\tilde\omega_1,\tilde\omega_2)$ is a generic element of the product of the old space with itself.}}
We set
\begin{align*}
		\RR_i(z,z',\bs,\tilde\omega) := F_i(\xi(z,z',\bs,\tilde\omega))
\end{align*}
and
\begin{align*}
	\RR'_i(z,z',\bs,\tilde\omega) :=
	\begin{cases}
		F_i(\xi(z,z',\bs,\tilde\omega))
		& \text{if } z = z', \\
		F_i(\xi'(z,z',\bs,\tilde\omega))
		& \text{if } z \neq z' \text{ both in } B(\hat x, r), \\
		F_i(\xi''(z,z',\bs,\tilde\omega))
		& \text{if } z \neq z' \text{ not both in } B(\hat x, r)
	\end{cases}
\end{align*}
for each $(z,z',\bs,\tilde\omega)\in \rr^d\times\rr^d\times (\rr_+)^m\times\tilde\Omega$ and~$i = 1, \dotsc, m$.
Now, let $\mathcal{E}^m_\la$ be the $m$-fold direct product of exponential laws with rate parameter~$\lambda$. {We denote by~$(\Omega, \mathcal{F}, \IP_{(x,x')})$  the direct product
of the probability space~$(\rr^d \times \rr^d, \mathcal{B}(\rr^d) \times \mathcal{B}(\rr^d), \delta_x \times \delta_{x'})$ with countably many copies of the probability space
\[
	 ((\rr_+)^m\times\tilde\Omega,\mathcal{B}((\rr_+)^m)\times\tilde{\mathcal{F}},\mathcal{E}^m_\la\times\tilde{\IP}),
\]
and define
the process $(z_k(\omega),z'_k(\omega))_{k\in\nn}$ inductively. First, set~$(z_0(\omega),z'_0(\omega)) = (y,y')$ where $\omega = (y,y',\omega_0, \omega_1, \dotsc) \in \Omega$ with $\omega_j=(\bs_j,\tilde \omega_j)\in (\rr_+)^m\times\tilde\Omega$,
$j = 0, 1, 2, \dotsc$, and $i = 1,\dotsc, m$. Then,}
\begin{align*}
	z_{jm+i}(\omega) &:= \RR_i(z_{jm}(\omega),z'_{jm}(\omega),\bs_j,\tilde \omega_j),\\
	z'_{jm+i}(\omega) &:= \RR'_i(z_{jm}(\omega),z'_{jm}(\omega),\bs_j,\tilde \omega_j).
\end{align*}
By construction, the pair~$( z_k,z'_k)$, $k \in \nn$ is a coupling for the embedded process:
\begin{equation}
\label{eq:indeed-hat-coupling}
	\IP_{(x,x')}\{\omega \in \Omega: z_k \in \Gamma\} = \hat{P}(x,\Gamma)
	\qquad\text{and}\qquad
	\IP_{(x,x')}\{\omega \in \Omega: z'_k \in \Gamma\} = \hat{P}(x',\Gamma)
\end{equation}
for all measurable~$\Gamma \subseteq \rr^d$.

We now state and prove two important properties of the constructed  coupling. The first one relies on~\textnormal{(C3)} and elucidates the choice of a construction by blocks of~$m$ steps with~$m$ as in Lemma~\ref{lem:minor-tot-var}. The second combines this first property and some technical consequences of Conditions~\textnormal{(C1)} and~\textnormal{(C2)} proved in Appendix~\ref{app:hitting} to establish an estimate on the time~$K$ needed for the coupling to hit the diagonal, {i.e.}~for the two coupled components to coincide; see~\eqref{eq:def-simdiag}. This will be crucial in the proof of the Main Theorem.
\begin{proposition}\label{prop:sc-to-coupling}
		There is a number $\hat p \in (0,1)$ such that
		\begin{equation}
			 \IP_{(x,x')}\left\{z_{ m} \neq z'_{ m} \right\} < \hat p
			 		\end{equation}
		for all $x,x' \in B(\hat x, r)$.
\end{proposition}

\begin{proof}
		With $\rectm$ as in and Lemma~\ref{lem:minor-tot-var}, the equality~\eqref{eq:coupl-tot-var} gives
		\begin{align*}
			& (\mathcal{E}^m_\la \times\tilde\IP)\left\{(\bs,\tilde\omega) : F_m(\xi(x,x',\bs,\tilde\omega)) \neq F_m(\xi'(x,x',\bs,\tilde\omega)) \right\} \\
				&\qquad \leq \mathcal{E}^m_\la(\rectm) \sup_{\bs  \in \rectm} \tilde{\IP}\left\{\tilde\omega : F_m(\xi(x,x',\bs,\tilde\omega)) \neq F_m(\xi'(x,x',\bs,\tilde\omega))\right\}  + \left(1 - \mathcal{E}^m_\la(\rectm)\right)\\
				&\qquad = \mathcal{E}^m_\la(\rectm) \sup_{\bs  \in \rectm} \|  F_m(x,\bs,{\cdot\,})_*(\ell^m) -   F_m(x',\bs,{\cdot\,})_*(\ell^m)\|_{\textnormal{var}} + \left(1 - \mathcal{E}^m_\la(\rectm)\right)
		\end{align*}
		whenever~$x$ and~$x'$ are in the ball~$B(\hat x, r)$. Therefore,
		\begin{align*}
			 \IP_{(x,x')}\left\{z_{ m} \neq z'_{ m} \right\}  \leq 1 - \mathcal{E}^m_\la(\rectm)( 1 - p)=: \hat p
		\end{align*}
		by Lemma~\ref{lem:minor-tot-var}.
\end{proof}
\begin{proposition}\label{prop:coincide}
		There are positive constants~$\theta_1$ and~$A_1$ such~that
		\begin{equation}
			\E_{(x,x')} \Exp{\theta_1  {\simdiag}} \leq A_1\left(1 + \|x\| + \|x'\|\right)
		\end{equation}for all $x,x' \in \rr^d$.
\end{proposition}
\begin{proof}
		Under Condition~\textnormal{(C1)},     $(x,x') \mapsto 1 + \|x\|^2 + \|x'\|^2$ is a   Lyapunov function for the coupling~$(z_k,z_k')_{k\in\nn}$. As a consequence of this, we control an exponential moment of the number~$\simcompact$ of jumps needed to enter a ball of large radius~$R$ around the origin  (see Corollary~\ref{lem:R4}).
		On the other hand, Condition~\textnormal{(C2)} guarantees the existence of a number~$M \in \nn_m$ of jumps in which transition probabilities from points in~$B(0,R)$ to the ball~$B(\hat x, r)$ are uniformly bounded from below  (see Lemma~\ref{lem:coupl-sync-near-hat}).

		Combining these results, we get the following bound on an exponential moment of the first simultaneous hitting time of the ball~$B(\hat x, r)$:
		there exist positive constants $\theta_2$ and~$A_2$ such that
		\begin{equation}\label{eq:hit-near-hat-couple}
			\E_{(x,x')} \Exp{\theta_2 \simball} \leq A_2\left(1 + \|x\|^2 + \|x'\|^2\right).
		\end{equation}
	 	This is stated and proved as Proposition~\ref{lem:hit-near-hat-couple} in the first appendix.
		Then, we   introduce a sequence of random times defined inductively by~$\simball_0 := 0$ and
		\[
			\simball_i := \min\left\{j \in \nn_m: z_j, z'_j \in B(\hat x, r) \text{ and } j > \simball_{i-1} \right\}
		\]
		for $i \geq 1$. Using the strong Markov property and applying the inequality~\eqref{eq:hit-near-hat-couple} repeatedly gives
		\begin{equation}\label{E:temps}
			\E_{(x,x')} \Exp{\theta_2 \simball_i} \leq \E \left(\Exp{\theta_2 \simball_{i-1}} \E_{(z_{\simball_{i-1}},z'_{\simball_{i-1}})} \Exp{\theta_2 \simball_1}\right) \leq {\hat C}^i \left(1 + \|x\|^2 + \|x'\|^2\right)
		\end{equation}
		for some positive constant ${\hat C}$.

		Note that Proposition~\ref{prop:sc-to-coupling} implies that~${\simdiag}$ is almost surely finite for all $x,x'\in \rr^d$. Indeed,
 		\begin{align}
			 \IP_{(x,x')}\{ {\simdiag} > \simball_i \}
				& \leq \IP_{(x,x')}\left\{z_{\simball_i + m} \neq z'_{\simball_i + m}\right\}\nonumber \\
				& = \IP_{(x,x')}\left(\left\{z_{\simball_i + m} \neq z'_{\simball_i + m} \right\} \big| \left\{z_{\simball_i} \neq z'_{\simball_i}\right\}\right) \IP_{(x,x')}\left\{z_{\simball_i} \neq z'_{\simball_i}\right\}\nonumber\\
				& \leq \hat p\, \IP_{(x,x')}\left\{z_{\simball_i} \neq z'_{\simball_i}\right\} \nonumber\\
				& \leq \hat p\, \IP_{(x,x')}\left\{z_{\simball_{i-1} + m} \neq z'_{\simball_{i-1} + m}\right\} \nonumber\\
				& \leq \hat p^i \label{E:26}
		\end{align}
		and almost-sure finiteness follows from the Borel--Cantelli lemma.
		Now, by H\"older's inequality,
		\begin{align*}
			 \E_{(x,x')} \Exp{\theta_1 {\simdiag}}
				  &  \leq 1 + \sum_{i=0}^\infty \E_{(x,x')} \left(\one_{\{\simball_i < {\simdiag} \leq \simball_{i+1}\}} \Exp{\theta_1 \simball_{i+1}}\right)
				\\ &  \leq 1 + \sum_{i=0}^\infty \left( \IP_{(x,x')}\{ {\simdiag} > \simball_i \}\right)^{1-\frac 1q} \left(\E_{(x,x')}\Exp{q \theta_1 \simball_{i+1}} \right)^{\frac 1q}
		\end{align*}
		for any $q \geq 1$. In each summand, the first term is controlled by the inequality~\eqref{E:26} and the second one
		   by~\eqref{E:temps}, provided that $ \theta_1\le \theta_2/q$:
		   \[
			\E_{(x,x')} \Exp{\theta_1 {\simdiag}}
				      \leq 1 + {\hat C}^{\frac 1q}\hat p^{\frac{1}{q}-1}\left(1 +\|x\|^2 + \|x'\|^2\right)^{\frac 1q} \sum_{i=0}^\infty \left({\hat C}^{\frac 1q} \hat p^{1-\frac{1}{q}}\right)^i.
			\]
		The proposition follows by taking $q \geq 2$ large enough that ${\hat C}^{\frac 1q} \hat p^{1-\frac{1}{q}}<1$.
\end{proof}

\subsection{Coupling for the original continuous-time process}

{Let the probability space~$(\Omega,\mathcal{F},\IP_{(x,x')})$ and the process $(z_k,z_k')$ be as in the previous subsection.
Recall that an element~$\omega$ of~$\Omega$ consists in an initial condition in~$\rr^d \times \rr^d$ and a sequence $(\bs_j,\tilde \omega_j)_{j\in\nn}$ of elements in~$(\rr_+)^m \times \tilde{\Omega}$ for some other probability space~$\tilde{\Omega}$ we have constructed. Let~$\tau_{jm+i}(\omega)$ be the positive real obtained by summing all the entries of~$\bs_1, \bs_2, \dotsc, \bs_j$ and the first~$i$ entries of $\bs_{j+1}$.}
Then, it follows from the construction of~$\IP_{(x,x')}$ that the sequence~$(\tau_k)_{k\in\nn}$ of random variables on~$(\Omega,\mathcal{F},\IP_{(x,x')})$ has independent increments distributed according to an exponential distribution with rate parameter~$\lambda$.

We define
\[
	Z_t(\omega) :=
		\begin{cases}
			z_{k}(\omega)  & \text{if } t = \tau_k(\omega), \\
			S_{t-\tau_k(\omega)}(z_k(\omega)) & \text{if } t \in (\tau_k(\omega),\tau_{k+1}(\omega))
		\end{cases}
		\] and \[
		Z'_t(\omega) :=
			\begin{cases}
				z'_{k}(\omega)  & \text{if } t = \tau_k(\omega), \\
				S_{t-\tau_k(\omega)}(z'_k(\omega)) & \text{if } t \in (\tau_k(\omega),\tau_{k+1}(\omega)).
			\end{cases}
\]
Then, \eqref{eq:disc-to-cont},~\eqref{eq:def-P-hat}	 and~\eqref{eq:indeed-hat-coupling} imply that~$(Z_t,Z'_t)$ is a coupling of $X_t$ and~$X'_t$.
\begin{proposition}\label{prop:almost-main}
	Under Conditions~\textnormal{(C1)--(C3)}, there exist positive constants $C$ and $c$ such that
	\begin{equation}
			\IP_{(x,x')}\{\stopping   > t\}
			\leq  C (1 + \|x\|+\|x'\|) \Exp{-ct}
	\end{equation}
	for any   $x,x'\in \rr^d$ and   $t \geq 0$.
\end{proposition}
\begin{proof}
	Let ${\simdiag}$ be defined by~\eqref{eq:def-simdiag}.
	As~$\tau_k$ is a sum of~$k$ independent exponentially distributed random variables with parameter~$\lambda$, the expectation of~$\Exp{2 c \tau_k}$ can be computed explicitly for~$c$ in the interval~$(0,\tfrac 12 \lambda)$, and $\tau_{\simdiag}$ is also almost-surely finite.
	For such a number~$c$, the Cauchy--Schwarz inequality yields
	\begin{align*}
		\E_{(x,x')} \Exp{c \tau_{\simdiag}} &= \sum_{k=0}^\infty \E_{(x,x')} \left(\Exp{c \tau_k} \one_{\{{\simdiag} = k\}}\right) 
		 	\leq \sum_{k=0}^\infty \left(\E_{(x,x')} \Exp{2 c \tau_k}\right)^{\frac 12} \left(\IP_{(x,x')}{\{{\simdiag} = k\}}\right)^{\frac 12}.
	\end{align*}
	On the other hand, we control $\IP_{(x,x')}{\{{\simdiag} \geq k\}}$ by Proposition~\ref{prop:coincide} and Chebyshev's inequality. Therefore,
	\begin{align*}
		\E_{(x,x')} \Exp{c \tau_{\simdiag}}
			&\leq \sum_{k=0}^\infty \left(\frac{\lambda}{\lambda - 2c}\right)^{\frac k2} \left( \Exp{-\theta_1k}A_1(1+\|x\|+\|x'\|)\right)^{\frac 12} \\
			&\leq A_1^{\frac 12}(1+\|x\|+\|x'\|)\sum_{k=0}^\infty \left(\frac{\lambda \Exp{-\theta_1}}{\lambda - 2c}\right)^{\frac k2},
	\end{align*}
	where $\theta_1$ and $A_1$ are as in Proposition~\ref{prop:coincide}. The series will converge for $c > 0$ small enough; fix such a value of~$c$. By Chebyshev's inequality, we find~$C > 0$ such that
	\[
				\IP_{(x,x')}\{\tau_{\simdiag} > t\}
				\leq  C (1+\|x\|+\|x'\|) \Exp{-c t}
	\]
		for all $x, x' \in \rr^d$.
		By construction, we have
	$
				\stopping   \leq \tau_{{\simdiag}}
	$
	almost surely and therefore
	\[
				\IP_{(x,x')}\{\stopping > t\}
					\leq C (1+\|x\|+\|x'\|) \Exp{-c t}.
	\]
	This completes the proof of the proposition.
\end{proof}

\subsection{Concluding the proof of the Main Theorem}

In view of Lemma~\ref{lem:existence}, {if we can find constants $C > 0$ and $c> 0$ such~that
\begin{align*}
	\|\PPPP_t^* \delta_x - \PPPP_t^* \delta_{x'} \|_\textnormal{var}
	&\leq C (1+\|x\|+\|x'\|) \Exp{-c t}
\end{align*}
for all $x,x' \in \rr^d$ and all $t \geq 0$, then integrating in~$x$ against~$\mu$ and in~$x'$ against $\mu^\textnormal{inv}$ gives the desired bound~\eqref{0.7A} with a different constant~$C$.}
By construction of the coupling $(Z_t,Z'_t)_{t \geq 0}$, we have
\begin{equation*}
	(\PPPP_t g)(x) - (\PPPP_t g)(x') = \E_{(x,x')} \left( g(Z_t) - g(Z'_t)  \right)
\end{equation*}
for all~$g \in L^\infty(\rr^d)$. Therefore,
\begin{align*}
	\|\PPPP_t^* \delta_x - \PPPP_t^* \delta_{x'} \|_\textnormal{var}
	&= \frac12\sup_{ \|g\|_\infty\le 1}      |(\PPPP_t g)(x) - (\PPPP_t g)(x')|  \\
	&\leq \frac12\sup_{ \|g\|_\infty\le 1}     \E_{(x,x')} | g(Z_t) - g(Z'_t)  |
\\&
= \frac12\sup_{ \|g\|_\infty\le 1}     \E_{(x,x')} \left\{\one_{\{Z_t \neq  Z'_t \}}| g(Z_t) - g(Z'_t)  |\right\}
\\
	&\leq      \IP_{(x,x')}\{ Z_t \neq  Z'_t \}
	\leq \IP_{(x,x')}\{ \stopping > t\}
\end{align*}
for all $x,x' \in \rr^d$ and   $t \geq 0$, and the result follows from Proposition~\ref{prop:almost-main}.

\section{Applications}

In this section, we apply the Main Theorem to the Galerkin approximations of \textsc{pde}s and to stochastically driven quasi-harmonic networks. For the Galerkin approximations we give a detailed derivation of   the controllability conditions and in the case of the networks we appeal to the results   obtained in~\cite{Ra18}. Before we do so, we briefly discuss the solid controllability assumption~\textnormal{(C3)}.

\subsection{Criteria for solid controllability}\label{A:B}

The notion of solid controllability was introduced by  Agrachev and Sarychev  in~\cite{AS05} (see also the survey~\cite{AS-2008}) in the context of the controllability of the 2D Navier--Stokes and Euler systems. It has been used in~\cite{AKSS07} to prove the existence of density for finite-dimensional projections of the laws of the solutions of randomly forced \textsc{pde}s. In~\cite{Sh17}, solid controllability is used to establish exponential mixing for some  random dynamical systems in a compact space, and in~\cite{Ra18}, for some classes of quasi-harmonic networks of oscillators driven by a degenerate Brownian motion. {It is the degeneracy allowed by this condition which sets our work apart from previous works on \textsc{sde}s driven by compound Poisson processes (that are too numerous to be cited here).}

{We compare it to two related well-known properties,} which might be more straightforward to check in some applications.
\begin{itemize}
 	\item[(C$3'$)]   \textit{Continuous  exact  controllability  from}~$\hat x$: there exists a nondegenerate closed ball $D \subset \rr^d$, a time $T_0>0$,  and a continuous function $\Psi: D \to C([0,T_0];\rr^n)$ such that $S_{T_0}(\hat x, \Psi(x)) = x$ for all $x \in D$.
 	\item[(C$3''$)] \textit{Weak H\"ormander condition at}~$\hat x$: the vector space spanned by the    family of vector fields
 	\begin{equation}\label{Lie-space}
 		 \left\{V_0,\, [V_1, V_2], \,\, [V_1,[V_2,V_3]],\,\, \dotsc : V_0 \in \mathbb{B},\,\, V_1, V_2,\dotsc \in \mathbb{B}\cup\{f\}\right\}
 	\end{equation}
  at the point~$\hat x$ coincides with $\rr^d$, where~$\mathbb{B}$ is  the set of constant vector fields formed by the columns of the matrix~$B$  and $[U,V](x)$ is  the \textit{Lie bracket}  of  the vector fields~$U$ and~$V$ in the point~$x$:
 	\[
 		[U,V](x)=D V(x)U(x)-DU(x)V(x).
 	\]
 Here, $DU(x)$ is
 the Jacobian matrix of~$U$ at~$x$.
\end{itemize}
It is shown in~\cite[\S{2.2}]{Sh17} that~\textnormal{(C$3''$)} implies~\textnormal{(C$3'$)} with arbitrary~$T_0$, and that~\textnormal{(C$3'$)} in turn implies~\textnormal{(C3)} with the same~$T_0$; see also~\cite[\S{3.2}]{Ra18}. The first implication appeals to some ideas from geometric control theory. The second implication can be seen from a degree theory argument (or alternatively from an application of Brouwer's fixed point theorem).

The weak H\"ormander condition, also known as the parabolic H\"ormander condition, has many important applications both in control theory (e.g., see~\cite[Ch.~5]{J-1997}) and  stochastic analysis (e.g., see \cite[\S{2.3} in Ch.~2]{nualart2006} and~\cite{Ha11}).  It is often assumed to hold in \textit{all} points of the state space. For finite-dimensional control systems, it ensures the global exact controllability; for It\^o diffusions, it guarantees existence and smoothness of the density of solutions with respect to the Lebesgue measure\,---\,a major step towards proving important ergodic properties. We emphasize that we bypass the study of smoothing properties of the transition function  of our Markov process and that the conditions stated need only hold in \textit{one} point of the state space (where Condition~\textnormal{(C2)} is also satisfied).

Recall that a pair of matrices, $A : \rr^d \to \rr^d$ and $B : \rr^n \to \rr^d$, is said to satisfy the {Kalman condition} if any $x\in \rr^d$ can be written as
$
	x = By_0 + ABy_{1} + \dotsb + A^{d-1}By_{d-1}
$
for some~$y_0, \dotsc, y_{d-1} \in \rr^n$.
For a linear control system of the form $\dot X = AX + B\zeta$, the Kalman condition implies~\textnormal{(C$3''$)} in all points through a straightforward computation of the Lie brackets; see~\cite[\S{1.2--1.3}]{Cor} for other well-known implications.
When~$f$ is a linear vector field~$x \mapsto Ax$ plus a perturbation, Condition~\textnormal{(C$3''$)} can be deduced at a point~$\hat x$ far from the origin by perturbing the Kalman condition on the pair~$(A,B)$, provided that one has good control on the decay of derivatives of the perturbation along a sequence of points~\cite[\S{5}]{Ra18}.

\subsection{Galerkin approximations of randomly forced PDEs}

In this section, we apply the Main Theorem to the Galerkin approximations of the following parabolic \textsc{pde} on the  torus $\T^D:=\rr^D/2\pi\zz^D$:
\begin{equation} \label{E:PDE}
	\partial_t u(t,x)-\nu\Delta_x u(t,x)+F(u(t,x))=h(x)+\ASc(t,x),   \quad  x\in   \T^D,
\end{equation}
where~$\nu>0$ is a constant, $h: \T^D\to \R$ is a given smooth function, and~$F:\R\to \R$ is a   function of the form
\begin{equation}\label{E:non}
F(u)=a  u^p+g(u).
\end{equation}We assume that $a>0$ is an arbitrary constant,   $p\ge3$ is an odd integer,  and $g:\R\to \R$ is a smooth function satisfying the following two conditions\footnote{The results of this subsection remain true under   weaker assumptions on the function $g$. This setting is chosen for the simplicity of   presentation.}:
\begin{description}
\item[(i)] there is a constant  $C>0$ such that
\[
	|g(u)|\le C (1+|u|)^{p-1}
\]
for all $u\in \R$.

\item[(ii)]	with $g^{(p)}$ the $p$-th derivative of $g$, the following limit holds
\[
	\lim_{u \to \pm \infty} g^{(p)}(u) = 0.
\]
\end{description}
For any $N\in \nn$,  consider the following finite-dimensional subspace of $L^2(\T^D)$:
\[
	\HH_N:=\text{span}\{s_k,\, c_k : k\in \zz^D,\,\, |k|\le N\},
\]
where $s_k(x):=\sin\lag x,k\rag$, $c_k(x):=\cos\lag x,k\rag$, $\lag x,k\rag:=x_1 k_1+\ldots+x_D k_D$ and $|k| := |k_1|+\ldots+|k_{ D}|$ for any multi-index~$k=(k_1, \ldots, k_{ D})\in \zz^D$ {and any vector~$x \in \T^D$}.
In particular,~$c_0$ is the constant function~$1$.
This subspace is endowed with the scalar product $\lag\cdot,\cdot\rag_{L^2}$  and the norm $\|\cdot\|_{L^2}$     inherited from~$L^2(\T^D)$.
Let ${\mathsf P}_N$ be  the orthogonal projection onto $\HH_N$ in $L^2(\T^D)$. The Galerkin approximations of~\eqref{E:PDE} are given~by
\begin{equation} \label{E:GPDE}
	\dot u(t)-\nu\Delta u(t)+{\mathsf P}_N F(u(t))=h+\ASc(t),
\end{equation}
where~$u$ is an unknown~$\HH_N$-valued function,~$h$ is an arbitrary vector in~$\HH_N$ and~$\ASc$ is a continuous~$\HH_1$-valued function.

Let us emphasize that the space~$H_1$ for the driving~$\ASc$ is the same for any level~$N\ge1$ of approximation, any value of the constant~$\nu$ and any function~$g$ satisfying~(i) and~(ii).

The main interest of the example considered in this section is that the perturbation term~$g$ in~\eqref{E:non} is quite general. In particular, we may have $F(u)=0$ in a large ball, so that the weak H\"ormander condition is not necessarily satisfied at all the points of the state space.
\begin{theorem}\label{T:4.1}
	Suppose that~\textnormal{(i)} and~\textnormal{(ii)} hold. Let $(\cpoiss_t)_{t\geq 0}$ be an $H_1$-valued compound Poisson with jump distribution~$\ell$ of finite variance and possessing a positive continuous density with respect to the Lebesgue measure on~$\HH_1$.
	Then, the semigroup~$(\PPPP^*_t)_{t \geq 0}$ for the \textsc{sde}
	\[
	\d u - \nu\Delta u \d t + {\mathsf P}_N F(u) \d t = h \d t + \d Y
	\]
	in~$\HH_N$ admits a unique invariant measure~$\mu\invar \in \mathcal{P}(\HH_N)$. Moreover, it is exponentially mixing in the sense that~\eqref{0.7A} holds for some constants~$C > 0$ and~$c > 0$, any  measure  $\mu\in\mathcal{P}(\HH_N)$, and any time~$t \geq 0$.
\end{theorem}
\begin{proof}
	The \textsc{sde} under consideration is of the form~\eqref{0.1} with $d=\dim\HH_N$, $n=\dim\HH_1=2D+1$, a smooth function $f_N:\HH_N\to \HH_N$  given by
 	\begin{equation}\label{E:NPDE}
 		f_N(u)= \nu \Delta u-{\mathsf P}_N F(u)+h,
 	\end{equation}
	and $B:\HH_1\to \HH_N$ the natural embedding operator.
	Let us show that    Conditions~\textnormal{(C1)}--\textnormal{(C3)} are verified.
	Using  the assumption~(i),  the fact that $s_k$  and $c_k$ are eigenfunctions of the Laplacian, and the Cauchy--Schwarz inequality, we get
	\begin{align*}
			\lag  f(u), u\rag_{L^2} &= \lag \nu \Delta u-{\mathsf P}_N F(u)+h, u\rag_{L^2} \\&\le - \nu\int_{\T^D} |u(x)|^2 \d x- C_1\int_{\T^D} |u(x)|^{p+1} \d x +C_2\\&\le-\nu\|u\|_{L^2} ^2 +C_2,
	\end{align*}where $C_1>0$ and $C_2>0$ are some constants and   $u\in \HH_N$ is arbitrary. This implies Condition~\textnormal{(C1)}.

	Condition~\textnormal{(C2)} (to all points) is a consequence of the global approximate controllability property  of Proposition~\ref{P:4.2} below, whose proof is given in Appendix~\ref{A:D}. Since it is proved in~\cite[\S{2.2}]{Sh17} that the weak H\"ormander condition implies solid controllability, Proposition~\ref{P:C3} below yields Condition~\textnormal{(C3)}.

	Thus, Conditions~\textnormal{(C1)}--\textnormal{(C3)} are satisfied and the proof of Theorem~\ref{T:4.1} is completed by applying our Main Theorem.
\end{proof}
\begin{proposition}\label{P:4.2}
	Equation~\eqref{E:GPDE} is approximately controllable: for any number~$\epsilon>0$,   any time~$T>0$, any initial condition $u_0\in \HH_N$, and any target $\hat u \in \HH_N$,   there exists a control~$\ASc \in C([0,T]; \HH_1)$ such that the solution $u$ of~\eqref{E:GPDE} with $u(0)=u_0$ satisfies
	\[
		\|u(T) - \hat u\|_{L^2}<\epsilon.
	\]
\end{proposition}
\begin{proposition}\label{P:C3}
There is a number $ R>0$ such that the	 weak H\"ormander Condition~\textnormal{(C$3''$)} is satisfied for equation~\eqref{E:GPDE} at any point $\hat u\in \HH_N$ with $\|\hat u \|_{L^2}\ge  R$.
\end{proposition}
\begin{proof}[Proof of Proposition~\ref{P:C3}]
	In view of the weak H\"ormander condition, we are interested in the nested subspaces $\{\VV_i\}_{i\geq 0}$ of $\HH_N$ defined by $\VV_0 = \HH_1$ and
	\[
			\VV_{i+1}(\hat u) := \operatorname{span}(\VV_{i} \cup \{[V, f_N](\hat u) : V \in \VV_{i}(\hat u)\}),
	\]
	where we at times identify the vector~$V \in \VV_i(\hat u)$ with the corresponding constant vector field on~$\HH_N$. Clearly, showing that $\VV_i(\hat u)=\HH_N$ for some $i$ large enough shows that the weak H\"ormander condition~\textnormal{(C3$''$)} holds in~$\hat u$.  We show in two steps that, indeed, $\VV_{(N-1)p}(\hat u)=\HH_N$ if~$\|\hat u\|_{L^2}$ is sufficiently large.

	\smallskip
	\noindent\textit{Step 1: Polynomial nonlinearity.}
		In this step, we assume that $g\equiv 0$, so that
		\begin{equation}\label{E:field}
		 		f_N(u)= \nu \Delta u-a  {\mathsf P}_N  (u^p)+h.
		\end{equation}
		In this case, Lie brackets with constant vector fields are especially straightforward to compute because~$\Delta$ is a linear operator and~$h$ is a constant vector. In particular, for any constant vector fields~$V_1$, $\dotsc$, $V_{p-2}$, $V_{p-1}$ and $V_p$,
		\begin{equation}\label{eq:pth-Lie-bracket}
			[V_1,\dotsc[V_{p-2},[V_{p-1},[V_p,f_N]]]\dotsc](\hat u) = -a \, p! \, \mathsf{P}_N(V_1 \dotsb V_{p-2} V_{p-1} V_p),
		\end{equation}
		where the product $V_1 \dotsb V_{p-2} V_{p-1} V_p$ is understood as a pointwise multiplication of functions.

		We claim that, for each multi-index~$m$ with $0 < |m| \leq N$, the vectors~$c_m$ and~$s_m$ are in~$\VV_{(|m|-1)p}(\hat u)$ for all~$\hat u \in \HH_N$.
		To start, note that if~$|l|\leq 1$, then~$c_l$ and~$s_l$ are in~$\HH_1$ and thus in~$\VV_i(\hat u)$ for each~$i$.

		Suppose now that~$c_m$ and~$s_m$ are in~$\VV_{(|m|-1)p}(\hat u)$. As noted above, for all multi-indices~$l$ with~$|l| \leq 1$, the vectors~$c_l$ and~$s_l$ are also in~$\VV_{(|m|-1)p}(\hat u)$. Therefore, combining the computation~\eqref{eq:pth-Lie-bracket} with trigonometric identities yields that
		\begin{align}
			\mathsf{P}_N c_{m \pm l}&= \mathsf{P}_N(1 \dotsb 1 \,c_lc_m) \mp \mathsf{P}_N(1 \dotsb 1 \,s_ls_m) \label{eq:proj-trigo-1} \\
				&= \tfrac{-1}{a \, p!}[c_0,\dotsc[c_0,[c_l,[c_m,f_N]]]\dotsc](\hat u) \pm \tfrac{1}{a \, p!}[c_0,\dotsc[c_0,[s_l,[s_m,f_N]]]\dotsc](\hat u)\notag
		\intertext{and}
			\mathsf{P}_N s_{m \pm l}&= \mathsf{P}_N (1 \dotsb 1 \,s_lc_m) \pm \mathsf{P}_N(1 \dotsb 1 \,c_ls_m) \label{eq:proj-trigo-2} \\
				&= \tfrac{-1}{a \, p!}[c_0,\dotsc[c_0,[s_l,[c_m,f_N]]]\dotsc](\hat u) \pm \tfrac{-1}{a \, p!}[c_0,\dotsc[c_0,[c_l,[s_m,f_N]]]\dotsc](\hat u)\notag
		\end{align}
		are in~$\VV_{(|m|-1)p+p}(\hat u)$. The result thus holds by induction on~$|m|$.

  \smallskip
  \noindent\textit{Step 2: The General case.}
		Let  $\tilde f_N$ be the vector field given by~\eqref{E:field}. If we consider the same Lie brackets as in Step 1, but now for the sum $\tilde f_N+{\mathsf P}_Ng$, the contribution of  ${\mathsf P}_Ng$ will vanish as $\hat u\to \infty$, thanks to assumption (ii). Therefore, $\VV_{(|N|-1)p}(\hat u)=\HH_N$, provided that~$\|\hat u\|_{L^2}$ is sufficiently large.
\end{proof}

\subsection{Stochastically driven networks of quasi-harmonic oscillators}\label{S:4.1}

Stochastically driven networks of oscillators play an important role in the investigation of various aspects of nonequilibrium statistical mechanics. In its simplest form, the setup can be described as follows. Consider $L$ unit masses, each labelled by an index $i \in \{1,\dotsc,L\}$ restricted to move in one dimension. Each of them is pinned by a spring of unit spring constant and, for $i \neq L$, the $i$th mass is connected to the $(i+1)$th mass by a spring of unit spring constant. The equations of motion for the positions and momenta, $(q_i,p_i)_{i=1}^L$, are the Hamilton equations
\[
\begin{cases}
		\d q_i = p_i \d t,  &
		 \d p_i = -(3q_i - q_{i-1} - q_{i+1})\d t, \qquad 1 < i < L,\\
		\d q_1 = p_1 \d t,  & \d p_1 = -(2q_{1} - q_{2})\d t,\\
		\d q_L = p_L \d t,  &  \d p_L = -(2q_{L} - q_{L-1})\d t.
\end{cases}
\]
Coupling the~$1$st [resp. the~$L$th] oscillator to a fluctuating bath with dissipation constant~$\gamma_1$ [resp.~$\gamma_L$] leads to the \textsc{sde}
\begin{equation}\label{eq:toy-network}
	\begin{cases}
			\d q_i = p_i \d t,
			&
			\d p_i = -(3q_i - q_{i-1} - q_{i+1}) \d t, \qquad\qquad\qquad 1 < i < L, \\
			\d q_1 = p_1 \d t,  & \d p_1 = -(2q_{1} - q_{2})\d t - \gamma_1 p \d t + \d Z_{1,t},\\
			\d q_L = p_L \d t,  & \d p_L = -(2q_{L} - q_{L-1})\d t  - \gamma_L p \d t + \d Z_{L,t},
	\end{cases}
\end{equation}
or variants thereof, where $Z_1$ and $Z_2$ are independent one-dimensional stochastic processes describing the fluctuations in the baths.

In the mathematical physics literature, many authors have considered nonlinear variants of this model where the thermal fluctuations\,---\,either acting on the momenta (the Langevin regime, as above) or on auxiliary degrees of freedom\,---\,are described by Gaussian white noise {i.e.}~$Z_{j,t} = \sqrt{2 \gamma_j \theta_j} W_{j,t}$, with~$W_{j,t}$ a standard Wiener process.
We refer the interested reader to~\cite{FKM65,Tr77}  for introductions to these models and discussions of their ergodic properties at thermal equilibrium; also see~\cite{JP97,JP98} for a generalization to non-Markovian models. The existence and uniqueness of the invariant measure is much more problematic out of equilibrium; see~\cite{LS77,EPR99a,EPR99b,EH00,RBT02,CEHRB}.
However, interesting phenomena pointed out in the physics literature for a single particle in a non-Gaussian bath~\cite{BC09,TC09,MQ+11,MG12} motivate a rigorous study of the mixing properties of corresponding networks. While the methods used for most of the previously cited existence and uniqueness results are not suitable to deal with compound Poisson processes, most of the ideas of~\cite{Sh17,Ra18} are. We develop the strategy to be followed in the present section.

Allowing for different spring constants and different ways of connecting the masses while staying in the Langevin regime leads us to considering the following generalization of~\eqref{eq:toy-network}.
Let~$I$ be a finite set and distinguish a nonempty subset $J \subset I$, where masses will be coupled to fluctuating baths. We use $\{\delta_i\}_{i\in I}$ [resp. $\{\delta_j\}_{j\in J}$] as the standard basis for~$\rr^I$ [resp.~$\rr^J$].
Let~$\omega: \rr^I \to \rr^I$ be a nonsingular linear map and let $\iota_j : \rr^J \to \rr^I$ be the rank-one map~$\delta_j \braket{\delta_j,{\cdot\,}}$ for each $j \in J \subset I$. The \textsc{sde}
\begin{align*}
	\d \begin{pmatrix}
		p \\ \omega q
	\end{pmatrix}
	&=
	\begin{pmatrix}
		- \sum_{j\in J} \gamma_j \iota_j\iota_j^* & -\omega^* \\
		\omega & 0
	\end{pmatrix}
	\begin{pmatrix}
		 p \\ \omega q
	\end{pmatrix} \d t
	+
	\sum_{j \in J}
	\begin{pmatrix}
		 \iota_j \\ 0
	\end{pmatrix} \d Z_j
\end{align*}
in~$\rr^{2|I|}$ then describes the positions~$q$ and momenta~$p$ of~$|I|$ masses connected to each other and pinned according to the matrix~$\omega$, with the $j$th oscillator being coupled to a Langevin bath with dissipation controlled by the constant~$\gamma_j > 0$ and fluctuations described by the process~$Z_j$.

In Proposition~\ref{prop:contr-Lang} and Corollary~\ref{cor:mix-Lang}, we consider a nonlinear version of this \textsc{sde} where the quadratic potential resulting form the springs is now  perturbed by a potential~$U : \rr^d \to \rr$. Their proofs are omitted since they are essentially the same as those of~Proposition~\ref{prop:contr-semi-M} and Corollary~\ref{cor:mix-semi-M} respectively. We start with dissipativity and controllability properties of the control system.
\begin{proposition}\label{prop:contr-Lang}
	Let $I,J,\omega$ and $(\gamma_j)_{j \in J}$ be as above. Then, the conditions
	\begin{enumerate}
		\item[\textnormal{(K)}] the pair $(\omega^*\omega,\sum_{j\in J} \iota_j\iota_j^*)$ satisfies the Kalman condition;
		\item[\textnormal{(G)}] the gradient of~$U$ is a smooth globally Lipschitz vector field growing strictly slower than $q \mapsto 1 + |q|^{\frac{1}{4|I|}}$;
		\item[\textnormal{(pH)}] there exists a sequence $\{q^{(n)}\}_{n\in\nn}$ of points in $\rr^I$, bounded away from~0, such that \[\lim_{n\to\infty} |q^{(n)}|^k \|D^{k+1}U(q^{(n)})\| = 0\]
		for each $k = 0, 1, \dotsc, d-1$;
	\end{enumerate}
	imply that the control system
	\begin{align*}
		\begin{pmatrix}
			\dot p \\ \omega \dot q
		\end{pmatrix}
		&=
		\begin{pmatrix}
			- \sum_{j\in J} \gamma_j \iota_j\iota_j^* & -\omega^* \\
			\omega & 0
		\end{pmatrix}
		\begin{pmatrix}
			 p \\ \omega q
		\end{pmatrix}
		-
		\begin{pmatrix}
			 \nabla U (q) \\ 0
		\end{pmatrix}
		+
		\sum_{j \in J}
		\begin{pmatrix}
			  \iota_j \\ 0
		\end{pmatrix} \zeta
	\end{align*}
	satisfies the conditions~\textnormal{(C1)},~\textnormal{(C2)} and~\textnormal{(C3)}.
\end{proposition}
The exponent in the formulation of the growth condition is typically not optimal; see~\cite{Ra18} for a formulation in terms of a power related to the Kalman condition.
The following mixing result for the corresponding \textsc{sde} with Poissonian noise essentially follows from our Main Theorem (see the proof of Corollary~\ref{cor:mix-semi-M}).
\begin{corollary}\label{cor:mix-Lang}
	Under the same assumptions, if $(N_j)_{j \in J}$ is a collection of~$|J|$ independent one-dimensional compound Poisson processes with jump distributions with finite variance and continuous positive densities with respect to the Lebesgue measure on~$\rr$, then the \textsc{sde}
	\begin{align*}
		\d \begin{pmatrix}
			p \\ \omega q
		\end{pmatrix}
		&=
		\begin{pmatrix}
			- \sum_{j\in J} \gamma_j \iota_j\iota_j^* & -\omega^* \\
			\omega & 0
		\end{pmatrix}
		\begin{pmatrix}
			 p \\ \omega q
		\end{pmatrix} \d t
		 -
		 \begin{pmatrix}
		 	 \nabla U (q) \\ 0
		 \end{pmatrix} \d t
		+
		\sum_{j \in J}
		\begin{pmatrix}
			 \iota_j \\ 0
		\end{pmatrix} \delta_j \d N_j
	\end{align*}
	admits a unique stationary measure~$\mu\invar \in \mathcal{P}(\rr^{I}\oplus\rr^{I})$. Moreover, it is exponentially mixing in the sense that~\eqref{0.7A} holds for some constants~$C > 0$ and~$c > 0$, any  measure  $\mu\in\mathcal{P}(\rr^{I}\oplus\rr^{I})$, and any time~$t \geq 0$.
\end{corollary}
In addition to the notation used so far, let $(\lambda_j)_{j \in J}$ be small positive numbers and let us use the shorthand~$\gamma\iota\iota^*$ for $\sum_j \gamma_j\iota_j\iota^*_j$, the shorthand~$\lambda\iota^*\iota$ for $\sum_j \lambda_j\iota_j^*\iota_j$, and so on. The \textsc{sde}
\begin{align*}
	\d \begin{pmatrix}
	 r \\	p \\ \tilde\omega q
	\end{pmatrix}
	&=
	\begin{pmatrix}
		-  \gamma\iota\iota^* & \lambda \iota \iota^* & 0 \\
		-\lambda\iota^*\iota & 0 & -\tilde\omega^* \\
		0 & \tilde\omega & 0
	\end{pmatrix}
	\begin{pmatrix}
		 r \\ p \\  \tilde\omega q
	\end{pmatrix} \d t
	+
	\begin{pmatrix}
		 \sqrt{2\gamma \theta}\iota^*\iota  \\ 0 \\ 0
	\end{pmatrix} \d W
\end{align*}
can be derived as the effective equation for the positions~$q$ and momenta~$p$ of a network of~$|I|$ masses connected to each other and pinned according to the matrix~$\omega$, with the $j$th oscillator being coupled to a classical Gaussian field at temperature~$\theta_j$ under some particular conditions on the coupling; see~\cite{EPR99a}. The~$|J|$ auxiliary degrees of freedom~$r \in \rr^J$ are introduced to make the process Markovian. The parameters~$\lambda_j$ and~$\gamma_j$ describe the coupling and dissipation for the $j$th bath. Here, the matrix $\tilde\omega$ encodes an effective quadratic potential and is such that $\tilde\omega^*\tilde\omega = \omega^*\omega - \lambda^2 \iota\iota^*$ ($\lambda$ is small), where $\omega$ encodes the original quadratic potential.
\begin{proposition}\label{prop:contr-semi-M}
	Let $I,J,\omega$ and $(\gamma_j)_{j \in J}$ be as above. Then, for $(\lambda_j)_{j \in J}$ small enough, the conditions~\textnormal{(K)}, \textnormal{(G)} and~\textnormal{(pH)} as in the previous proposition imply that the the control system
	\begin{align*}
		\begin{pmatrix}
		 \dot r \\	\dot p \\ \tilde\omega \dot q
		\end{pmatrix}
		&=
		\begin{pmatrix}
			-  \gamma\iota\iota^* & \lambda \iota \iota^* & 0 \\
			-\lambda\iota^*\iota & 0 & -\tilde\omega^* \\
			0 & \tilde\omega & 0
		\end{pmatrix}
		\begin{pmatrix}
			 r \\ p \\  \tilde\omega q
		\end{pmatrix}
		-
		\begin{pmatrix}
			 0 \\ \nabla U (q) \\ 0
		\end{pmatrix}
		+
		\begin{pmatrix}
			 \one \\ 0 \\ 0
		\end{pmatrix} \zeta
	\end{align*}
	satisfies the conditions~\textnormal{(C1)},~\textnormal{(C2)} and~\textnormal{(C3)}.
\end{proposition}
\begin{proof}
	The Kalman condition on the pair~$(\omega^*\omega,\iota\iota^*)$ implies the Kalman condition on the pair $(\tilde\omega^*\tilde\omega,\iota\iota^*)$ if $\lambda$ is small enough. This in turn implies that the pair
	\[
	(A,B) := \left(
	\begin{pmatrix}
		-  \gamma\iota\iota^* & \lambda \iota \iota^* & 0 \\
		-\lambda\iota^*\iota & 0 & -\tilde\omega^* \\
		0 & \tilde\omega & 0
	\end{pmatrix}
	,
	\begin{pmatrix}
		 \one \\ 0 \\ 0
	\end{pmatrix}
	\right)
	\]
	also satisfies the Kalman condition; see Proposition~4.1 in~\cite{Ra18}. It follows  by Lemma~5.1(2) in~\cite{JPS17} that the eigenvalues of~$A$ then have strictly negative real part. {Combined with the growth assumption~\textnormal{(G)}, the negativity of the eigenvalues implies~\textnormal{(C1)} for a suitable inner product; see Lemma~3.1 in~\cite{Ra18}.} Proposition~3.3 in~\cite{Ra18} says that the Kalman condition on~$(A,B)$ and the growth condition~\textnormal{(G)} on~$\nabla U$ give~\textnormal{(C2)} everywhere.
	The fact that the Kalman condition on~$(A,B)$ and assumption~\textnormal{(pH)} give the weak H\"ormander condition~\textnormal{(C3'')} in one point is the content of Proposition~5.1 in~\cite{Ra18}. But, as previously mentioned, the weak H\"ormander condition implies solid controllability.
\end{proof}
Concerning the corresponding \textsc{sde} with Poissonian noise, we have the following mixing result\,---\,which again parallels that of~\cite{Ra18}\,---\,as a corollary of the controllability properties.
\begin{corollary}\label{cor:mix-semi-M}
	Under the same assumptions, if $(N_j)_{j \in J}$ is a collection of~$|J|$ independent one-dimensional compound Poisson processes with jump distributions with finite variances and continuous positive densities with respect to the Lebesgue measure on~$\rr$, then the \textsc{sde}
	\begin{align*}
		\d \begin{pmatrix}
		 r \\	p \\ \tilde\omega q
		\end{pmatrix}
		&=
		\begin{pmatrix}
			-  \gamma\iota\iota^* & \lambda \iota \iota^* & 0 \\
			-\lambda\iota^*\iota & 0 & -\tilde\omega^* \\
			0 & \tilde\omega & 0
		\end{pmatrix}
		\begin{pmatrix}
			 r \\ p \\  \tilde\omega q
		\end{pmatrix} \d t
		 -
		 \begin{pmatrix}
		 	 0 \\ \nabla U (q) \\ 0
		 \end{pmatrix} \d t
		+
		\begin{pmatrix}
			 \one \\ 0 \\ 0
		\end{pmatrix}  \sum_{j \in J} \delta_j \d N_j.
	\end{align*}
	admits a unique stationary measure~$\mu\invar \in \mathcal{P}(\rr^{J}\oplus\rr^{I}\oplus\rr^I)$. Moreover, it is exponentially mixing in the sense that~\eqref{0.7A} holds for some constants~$C > 0$ and~$c > 0$, any~$\mu\in\mathcal{P}(\rr^{J}\oplus\rr^{I}\oplus\rr^I)$, and any time~$t \geq 0$.
\end{corollary}

\begin{proof}[Proof sketch]
	If the noise $\sum_{j \in J} \delta_j  N_j  $ were replaced by a single compound Poisson process whose jump distribution possesses a finite second moment and a positive continuous density with respect to the Lebesgue measure on~$\rr^J$, then our Main Theorem would apply.

	Although the probability that jumps in the different baths occur simultaneously is zero by independence, there is a positive probability that they occur arbitrarily close to simultaneity. Since an independent sum of a jump from each distribution gives a random variable with a finite variance and a positive continuous density with respect to the Lebesgue measure on~$\rr^J$, our control arguments can be adapted using additional continuity arguments.
\end{proof}

\appendix

\section{Exponential estimates on hitting times}
\label{app:hitting}

In this appendix, we present results on hitting times for  the  coupling $(z_k,z_k')$  constructed in Subsection~\ref{S:3.1}.   Loosely speaking, estimates on the hitting times of a small ball near~$\hat x$ are obtained by combining a lower bound on the hitting time of a (large) compact   around the origin and a lower bound on the probability of making a transition from the aforementioned compact to the small ball. We shall assume that Conditions~\textnormal{(C1)}--\textnormal{(C3)} are satisfied and fix the parameters $m, r,$ and $p$ as in Lemma~\ref{lem:minor-tot-var}.

We provide an estimate for the first simultaneous hitting time~$I$ of a ball of large radius $R$ around the origin. To do this, we use the preliminary estimates of Lemma~\ref{lem:Lyap} to exhibit the existence of a suitable Lyapunov structure and conclude with a standard argument.

\begin{lemma}\label{pre-lem:R4}
	The function~$V$ defined by~$V(y,y'):= 1+\|y\|^2+\|y'\|^2$ is a Lyapunov function in the sense that there exist positive constants~$R$ and~$C_*$ and a constant~$0 < a < 1$ such that
	\begin{align}
			\E_{(x,x')} V(z_{m},z_{m}')&\le a \,V(x,x')\,\quad\quad \quad    \text{for $\|x\|\vee\|x'\|\ge R$,}\label{E:31}\\
			\E_{(x,x')} V(z_{k},z_{k}')&\le C_* \, \quad\quad\quad \quad\quad \quad  \text{for $\|x\|\vee\|x'\|<R$, $k\ge0$}\label{E:32}.
	\end{align}
\end{lemma}

\begin{proof}
	By Lemma~\ref{lem:Lyap}, there is $\gamma \in (0,1)$ such~that
	\begin{align}
		\E_{(x,x')} (1 + \|z_{k}\|^2 + \|z'_{k}\|^2) &= 1 + \E_x \|X_{\tau_k}\|^2 + \E_{x'} \|X_{\tau_k}\|^2\nonumber \\
			&\leq 1 + \gamma^k (\|x\|^2 + \|x'\|^2) + 2C(1 + \Lambda)\label{E:33}
	\end{align}
	for all $k \in \nn$ and   $x, x' \in \rr^d$.  Taking  $k=m$, any $a\in (\gamma^m,1)$, and any $x,x'\in \rr^d$ such that
	\[
	\|x\|\vee \|x'\| \geq   (a -\gamma^m)^{-1/2} (1-a +  2C(1 + \Lambda))^{1/2}=:R,
	\]
	we get
	\begin{align*}
		\E_{(x,x')} \left(1 + \|z_m\|^2 + \|z'_m\|^2\right) &\leq a \left(1+\|x\|^2 + \|x'\|^2\right).
	\end{align*}
	Thus,~\eqref{E:31} holds.
 	In the case
	$
		\|x\|\vee \|x'\| \leq   R,
	$
	by~\eqref{E:33}, we have
	\begin{align*}
		\E_{(x,x')} (1 + \|z_{k}\|^2 + \|z'_{k}\|^2) &\leq 1 +   2R^2 + 2C(1 + \Lambda)=:C_*.
	\end{align*}
	This gives~\eqref{E:32} and completes the proof of the lemma.
\end{proof}

{It is well known that the Lyapunov structure of the previous lemma implies a bound on an exponential moment for the time needed to reach a large enough level set of the Lyapunov function~$V$. While arguments for this implication can be found in~\cite{MeTw}, we give a brief proof sketch and refer the reader to Proposition~3.1 in~\cite{Sh08} for a statement and complete proof which more precisely reflects our approach.}

\begin{corollary}\label{lem:R4}
		There exist positive constants $R$, $c_1$, and $C_1$ such that
		\[
			\E_{(x,x')}\Exp{c_1 \simcompact} \leq C_1(1 + \|x\|^2 + \|x'\|^2)
		\]
		for all $x,x' \in \rr^d$,
		where
		\[
			\simcompact :=\min \{ j \in \nn_m^0  : z_j, z'_j \in B(0,R) \}
		\]
\end{corollary}

{
\begin{proof}[Proof sketch]
	One can show using the Markov property and~\eqref{E:31} repeatedly that
	\[
		\ee_{(x,x')}[\one_{\{I > n m\}} V(z_{nm},z'_{nm})] \leq a^n V(x,x')
	\]
	and deduce using $V \geq 1$ that
	\begin{equation}
	\label{eq:3.7-in-Sh08}
		\pp_{(x,x')}[I > n m] \leq a^n V(x,x').
	\end{equation}
	By~\eqref{eq:3.7-in-Sh08} and the Borel--Cantelli lemma, $I$ is almost surely finite. Therefore, one can use
	\begin{align*}
		\ee_{(x,x')} \Exp{c_1 I} \leq 1 + \sum_{n=1}^\infty \ee_{(x,x')}[\one_{\{I = nm\}} \Exp{c_1 I}]
	\end{align*}
	and, for~$c_1$ small enough, the right-hand side can be bounded using~\eqref{eq:3.7-in-Sh08} in terms of $V(x,x')$ and a convergent geometric series.
\end{proof}
}

In what follows $R$, $c_1$ and~$C_1$ will be as in Corollary~\ref{lem:R4}. We continue with another estimate on an exponential moment.

\begin{lemma}\label{lem:coupl-exp-fast-in-compact}
	For any   $M \in \nn$, there is a   constant $C_2 > 0$ such that
	\begin{equation}\label{E:sig_i}
		\E_{(x,x')} \Exp{c_1 \simcompact_i} \leq C_2^i (1 + \|x\|^2 + \|x'\|^2)
	\end{equation}
	for all $x, x' \in \rr^d$ and   $i\in \nn$, where
	$\simcompact_0:=0$ and
	\[
		\simcompact_i:= \min\left\{ j \in \nn_m : j \geq \simcompact_{i-1} + M \text{ and } z_{j},z'_{j} \in B(0,R)\right\}.
	\]
\end{lemma}

\begin{remark}
	The stopping time~$\simcompact_i$ depends on both~$M$ and~$R$. The value of~$R$ was already fixed in Corollary~\ref{lem:R4} and, in our application,~$M$ will be as in Lemma~\ref{lem:coupl-sync-near-hat}. It is important that the constant~$C_2$ does not depend on~$x$ and~$x'$.
\end{remark}

\begin{proof}
	  	By our last corollary, the Markov property, and~\eqref{eq:moment-emb} in Lemma~\ref{lem:Lyap}, we have
	\begin{align}
		\E_{(x,x')} \Exp{{c_1   \simcompact_1}}
			&= \Exp{c_1 M} \E_{(x,x')} \left( \E_{(z_M,z_M')}\Exp{c_1  \simcompact}\right) \nonumber\\
			&\leq C_1 \Exp{c_1 M} \E_{(x,x')} (1 + \|z_M\|^2 + \|z'_M\|^2) \nonumber\\
			&\leq C_1 \Exp{c_1 M}   (1 + \gamma^M \|x\|^2 + \gamma^M\|x'\|^2 + 2 C(1+ \Lambda))\nonumber \\
			&\leq \tilde C_1   (1 +\|x\|^2 +  \|x'\|^2)\label{E:40}
	\end{align}
	for $\tilde C_1$ a combination of $C$, $C_1$ and~$\Lambda$.
	In particular, for any   $x, x' \in B(0,R)$,
	\begin{align*}
		\E_{(x,x')} \Exp{{c_1    \simcompact_1}}
			 \leq \tilde C_1   (1 +R^2 +  R^2)=:C_2.
			 \end{align*}
			 Then
	 $z_{\simcompact_{i-1}},z'_{\simcompact_{i-1}}  \in B(0,R)$ for any $i > 1$,  and therefore
	 \[
			\E_{(x,x')} \Exp{c_1 \simcompact_i}  = \E_{(x,x')} \left(\Exp{c_1 \simcompact_{i-1}}  \E_{(z_{\simcompact_{i-1}},z'_{\simcompact_{i-1}})}\Exp{c_1  \simcompact_1 }\right)  			 \leq C_2 \E_{(x,x')}  \Exp{c_1 \simcompact_{i-1}}
			 \leq C_2^{i-1} \E_{(x,x')}  \Exp{c_1 \simcompact_{1}} .
		\]
	Finally, using~\eqref{E:40}, we obtain~\eqref{E:sig_i}.
\end{proof}

\begin{lemma}\label{lem:coupl-sync-near-hat}
	Consider the random variable
	\[
		\simball:= \min\left\{j \in \nn_m^0 : z_{j},z'_{j} \in B(\hat x, r)\right\},
	\]
	where $\hat x$ is as in Condition~\textnormal{(C2)}.
	There exists $M  \in \nn_m$ such that
		\begin{equation}\label{E:timectrl}
			0< q  := \inf_{x,x' \in B(0,R)} \IP_{(x,x')} \left\{\simball  \le  M\right\}.
		\end{equation}
\end{lemma}

\begin{proof}
	Let~$T$ be the time  in Condition~\textnormal{(C2)} for~$\epsilon = \tfrac r2 $ and radius $R$.  To simplify the presentation, we assume that   $T=1$.

 	\smallskip
 	\noindent \textit{Step 1: controlling a single trajectory of the \textsc{sde}~\eqref{0.1}.} First, let us show an inequality like~\eqref{E:timectrl} for a single trajectory of the \textsc{sde}~\eqref{0.1}.
 	Take an initial condition $x \in B(0,R)$. By Condition~\textnormal{(C2)}, there exists a control~$\zeta_x \in C([0,1];\rr^n)$ such that
	\begin{equation}\label{E:ctrl}
			\|S(x,\zeta_x) - \hat x\| < \frac r2.
	\end{equation}
	By a standard continuity and compactness argument, we can find a   finite set
	\(
		Z 
		\subset C([0,1];\rr^n)
	\)
	such that the control $\zeta_x$ in~\eqref{E:ctrl} can be chosen from  $Z$ for any $x\in B(0,R)$. For any integer~$M\ge1$, let the mapping $F_M:\rr^d\times (\rr_+)^M\times (\rr^n)^M \to \rr^d$ be defined by~\eqref{E:fk}, let $\iota_M$ be as in Lemma~\ref{lem:minor-tot-var}, and consider the sets
 	\begin{align*}
		\Delta &:= \left\{\bs=(s_j)_{j=1}^M\in (\rr_+)^M : s_j \in \left(\frac{1-\delta}{M},\frac{1}{M}\right),\quad  j = 1, \dotsc, M\right\},\\
		\Xi_x&:=\left\{\bxi=(\xi_j)_{j=1}^M\in(\rr^n)^M : \left\| \iota_M(\zeta_{x})  - \bxi\right\|_{(\rr^n)^M} <  \delta    ,\quad  j = 1, \dotsc, M\right\}
	\end{align*}for any $\delta>0$. Again by a   continuity and compactness argument, it is not hard to see that
	\begin{align*}
			\Delta \times  \Xi_x \subset  \left\{\bs\in (\rr_+)^M,\, \bxi\in (\rr^n)^M  : \|F_M(x,\bs,\bxi) - \hat x\| < r \right\}
	\end{align*}for sufficiently large $M\in \nn_m$, small $\delta>0$, and any $x\in B(0,R)$.
	Note that $F_M(x,\bs, \bxi)=X_{\tau_M}$ when $\bs=(t_j)_{j=1}^M$ and $\bxi=(\eta_j)_{j=1}^M$.
	By our assumptions on the laws of $t_j$ and $\eta_j$, it is clear that\footnote{Recall that $\mathcal{E}_\la^M$ and $\ell^M$ stand for the $M$-fold products of the exponential distribution and $\ell$, respectively.}
	\begin{gather*}
		   \mathcal{E}_\la^M (\Delta)=  \prod_{j = 1}^M  \left(\Exp{-\lambda  \frac{1-\delta}{M}}- \Exp{-\lambda \frac{1}{M}}\right)>0,\\
   \inf_{x\in B(0,R)}\ell^M(\Xi_x) > 0,
	\end{gather*}
	since there is only a finite number of sets $\Xi_x$ for $x$ in $B(0,R)$.
	We conclude that
	\begin{equation}\label{E:36}
		0<  \inf_{x\in B (0,R)} \IP_{x} \left\{ \|X_{\tau_M}-\hat x\|<r \right\}.
	\end{equation}

	\smallskip
	\noindent \textit{Step 2: case of coupling   trajectories.} We consider three cases.

	$\bullet$
		If $x = x'$, then the trajectories $z_{j}$ and $z'_{j}$ coincide for all $j$ and the result follows immediately from~\eqref{E:36}.

	$\bullet$   If $x \neq x'$ with $x, x' \in B(\hat x, r)$, then
			\[\IP_{(x,x')} \left\{\simball =0 \right\} = 1.\]

	$\bullet$
		If $x \neq x'$ not both in $B(\hat x, r)$, consider $\bs \in \Delta$, $\bxi \in \Xi_x$, and  $\bxi' \in \Xi_{x'}$. By construction, both~$F_M(x,\bs,\bxi)$ and $F_M(x',\bs,\bxi')$ lie in~$B(\hat x, r)$.
		Then, there exists a minimal $k \in \nn_m$ such that both $F_k(x,\bs,\bxi)$ and $F_k(x',\bs,\bxi')$ lie in~$B(\hat x, r)$. Necessarily,~$k$ satisfies $k \leq M$.
		Therefore,~the~construction of the coupling\,\footnote{When the coupling starts with $x \neq x'$ not both in $B(\hat x, r)$, the   first $m$ jumps are independent. The probability of $z_m = z'_m$ is zero by our assumptions on~$\ell$. Thus going by blocks of~$m$ steps, we see that the jumps are independent until both trajectories simultaneously hit $B(\hat x, r)$ at a time which is a multiple of~$m$.}
		implies that $z_k, z'_k$ are guaranteed to be in~$B(\hat x, r)$ for some $ k \leq M$ {for all~$\omega = (x,x',(\bs_j,\tilde\omega_j)_{j \in \nn})$ such that $(\bs_j)_{j=1}^{M/m}$} lies in~$\Delta$
		and such that $(\xi(x,x',\bs_j,\tilde\omega_j))_{j=1}^{M/m}$ and~$(\xi''(x,x',\bs_j,\tilde\omega_j))_{j=1}^{M/m}$ lie respectively in~$\Xi_x$ and~$\Xi_{x'}$. By construction,
		{
		\begin{align*}
			\tilde\IP\left\{\tilde\omega_j : \xi(x,x',\bs_j,\tilde\omega_j) \in
			 \Xi_x\right\}  &= \ell^M(\Xi_x), \\
			\tilde\IP\left\{\tilde\omega_j : \xi''(x,x',\bs_j,\tilde\omega_j) \in
			 \Xi_{x'}\right\} & = \ell^M(\Xi_{x'}),
		\end{align*}
		and
		\begin{equation*}
			\mathcal{E}_\la^M
			(\Delta)
			=  \prod_{j = 1}^M  \left(\Exp{-\lambda  \frac{1-\delta}{M}}- \Exp{-\lambda \frac{1}{M}}\right) .
		\end{equation*}}
		Then, independence gives
		\[
			  \IP_{(x,x')}\left\{\simball \leq M \right\} \\
			  \geq \ell^M(\Xi_x)\, \ell^M(\Xi_{x'})  \prod_{j = 1}^M  \left(\Exp{-\lambda  \frac{1-\delta}{M}}- \Exp{-\lambda \frac{1}{M}}\right) >0 .
		\]
  	The uniformity in~$x$ and~$x'$  follows from the fact that there  is  only   a finite number of sets~$\Xi_x$ and $\Xi_{x'}$ to consider as~$x$ and~$x'$ range over the set~$B(0,R)$.
\end{proof}

The main result of this appendix is the following exponential-moment bound on the random variable~$\simball$. {The argument used to deduce the proposition from the previous lemmas is well known and is for example discussed in depth in Section~3.3.2 in~\cite{KuSh}.} 
\begin{proposition}\label{lem:hit-near-hat-couple}
	There are constants~$\theta_2>0$ and $A_2>0$ such that
	\begin{equation}\label{1.4}
		\E_{(x,x')} \Exp{\theta_2  \simball} \le A_2\left(1+\|x\|^2 + \|x'\|^2\right)
 	\quad
 	\end{equation}
	for all $x, x' \in \R^d$.
\end{proposition}
\begin{proof}
	Let $\simcompact_i $ be defined as in Lemma~\ref{lem:coupl-exp-fast-in-compact} with constant  $M\in\nn_m$ as in  Lemma~\ref{lem:coupl-sync-near-hat}.
	Then
	\[
		 \IP_{(x,x')} \left\{\simball > k	\right\} \\
			  \leq \IP_{(x,x')} \{ \simcompact_i < \simball\} + \IP_{(x,x')} \{ \simcompact_i \geq k \}
	\]
	for any choice of integers $i,k\ge1$.
	To control the first term, note that the Markov property and Lemma~\ref{lem:coupl-sync-near-hat} imply
	\[
		 \IP_{(x,x')} \left\{ \simcompact_i <  \simball \right\}    \leq (1- q )\,\IP_{(x,x')} \left\{ \simcompact_{i-1} < \simball \right\}   \leq (1- q )^{i-1}.
	\]
	For the second term, we have the bound
	\begin{align*}
		\IP_{(x,x')} \{ \simcompact_i  \geq k \} \leq  C_2^{i} \Exp{-c_1 k}(1 + \|x\|^2 + \|x'\|^2)
	\end{align*}
	by Chebyshev's inequality and Lemma~\ref{lem:coupl-exp-fast-in-compact}.
	 In particular, taking $i$ scaling like $\epsilon k$ for $\epsilon$ small enough, we find
	\begin{align*}
		\IP_{(x,x')} \left\{ \simball > k
		\right\} & \leq (1- q )^{\epsilon k - 1} + C_2^{\epsilon k} \Exp{-c_1 k} (1 + \|x\|^2 + \|x'\|^2) \\
			&\leq C_3 a^k (1 + \|x\|^2 + \|x'\|^2)
	\end{align*}
	for some $a \in (0,1)$ and $C_3 > 0$. This exponential decay of the probability yields the proposition for~$\theta_2$ small enough and $A_2$ large enough.
\end{proof}

\section{Controllability of ODEs with polynomially growing nonlinearities}\label{A:D}

When the perturbation term~$g$ in~\eqref{E:non} is a polynomial, Proposition~\ref{P:4.2} follows  from~\cite[Thm.~3]{JK-1985} or~\cite[Thm.~11 in Ch.~5]{J-1997} and the system is even exactly controllable. In the general case, when~$g$ is an arbitrary smooth function satisfying~(i) and~(ii),  these results  cannot be  applied since the
H\"ormander con\-dition is not necessarily satisfied at all the points. We adapt an argument used in~\cite[Thm.~2.5]{N-2019} which is particularly simple in the case of ordinary differential equations. Let us consider the equation
\begin{equation} \label{E:GPDE2}
	\dot u(t)-\nu\Delta (u(t)+\AScaux(t))+{\mathsf P}_N F(u(t)+\AScaux(t))=h+\ASc(t),
\end{equation}
with two controls~$\AScaux$ and~$\ASc$ in~$C([0,T]; \HH_N)$.\footnote{The idea of introducing the second control $\AScaux$ comes from~\cite{AS05} and is nowadays extensively used in the control theory of PDEs with finite-dimensional controls (see the surveys~\cite{AS-2008, MR3777005}).} We denote by $S_t(u_0,\AScaux, \ASc)$ the solution of~\eqref{E:GPDE2} satisfying  the initial condition~$u(0)=u_0$. To simplify the presentation, we shall assume that~$a=1$ in~\eqref{E:non}.
Let us define a sequence $\{\HHH_i\}_{i \geq 1}$ of   subspaces of $H_N$ as follows: $\HHH_1=H_1$   and
\[
	\HHH_i=\text{span}\left\{{\mathsf P}_N(\ASconstcaux_1\cdot\ldots\cdot\ASconstcaux_p):\,\,\ASconstcaux_j\in\HHH_{i-1},\,\, j=0,\ldots,p  \right\}
\]
for $i\ge2$.
The trigonometric identities~\eqref{eq:proj-trigo-1} and~\eqref{eq:proj-trigo-2} give that
$s_{l\pm m}, c_{l\pm m} \in \HHH_i$, provided that $s_{l}, s_{ m}, c_{l}, c_{ m}\in \HHH_{i-1}$. Recalling the definition of $H_1$, it is easy to infer that
\begin{equation}\label{E:47}
\HHH_i=\HH_N \quad \text{ for sufficiently large $i\ge1$.}
\end{equation}
We will also use another form of these subspaces:
\begin{equation}
	\HHH_i=\text{span}\left\{\ASconstcaux_0,  \,\,{\mathsf P}_N \ASconstcaux^p:\,\,\ASconstcaux_0,\ASconstcaux\in\HHH_{i-1}   \right\} \label{E:48}
\end{equation}
for~$i\ge2$, which can be verified as in Lemma~4.2 in~\cite{N-2019}.

The following lemma will play an important role in the proof of Proposition~\ref{P:4.2}. It is established at the end of this subsection.
\begin{lemma}\label{B1}
	Under the conditions of Theorem~\ref{T:4.1},
	for any vectors $u_0,\ASconstcaux, \ASconstc\in \HH_N$, we have
	\begin{equation}\label{E:49}
 	 	S_{\de}(u_0,\de^{-1/p}\ASconstcaux,\de^{-1}\ASconstc)\to  u_0+\ASconstc-  {\mathsf P}_N\ASconstcaux^p \quad\text{in $\HH_N$ as $\de\to 0$}.
	\end{equation}
\end{lemma}
\begin{proof}[Proof of Proposition~\ref{P:4.2}]
	By a general argument (see for example Step~4 in the proof of Theorem~2.3 in~\cite{N-2019}) approximate controllability in any fixed time~$T>0$ can be obtained from controllability in arbitrarily small time.

	Lemma~\ref{B1} gives that for all~$u_0 \in \HH_N$, $\ASconstc \in \HH_1 = \HHH_1$, $\epsilon > 0$, and~$T>0$, there exists~$ \ASc \in C([0,\delta];\HH_1)$ with $0 < \delta < T$ such that
	\begin{equation}\label{eq:appr-goal}
		\|S_{\delta}(u_0, \ASc) - (u_0 + \ASconstc)\|_{L^2} < \epsilon.
	\end{equation}

	Because $\HH_N = \HHH_i$ for some $i$, we may proceed by induction on~$i$:
	let us suppose that for all~$u_0 \in \HH_N$, $\ASconstc \in \HHH_{i-1}$, $\epsilon > 0$, and~$T>0$, there exists $ \ASc \in C([0,\delta];\HH_1)$ with $0 < \delta < T$ such that~\eqref{eq:appr-goal} holds;
	we will show that this property then also holds for~$i$,
	and the proof of the proposition will be complete.

	Fix $u_0 \in \HH_N$. By~\eqref{E:48}, any~$\ASconstc \in \HHH_i$ can be written as a linear combination of elements of the form $\mathsf{P}_N \ASconstcaux^p$ with~$\ASconstcaux \in \HHH_{i-1}$, plus a vector in~$\HHH_{i-1}$.
	Hence, by an iteration argument, it suffices to consider vectors~$\ASconstc$ of the form~$-\mathsf{P}_N \ASconstcaux^p$ for some $\ASconstcaux \in \HHH_{i-1}$. Let $\epsilon > 0$ and $T>0$ be arbitrary. By Lemma~\ref{B1}, there exists $\delta_2 \in (0,\tfrac 13 T)$ such that
	\[
		\|S_{\delta_2}(u_0,\delta_2^{-1/p}\ASconstcaux,0) - (u_0 - \mathsf{P}_N\ASconstcaux^p)\|_{L^2} < \tfrac 14 \epsilon.
	\]
	On the other hand, a change of variable shows
	\[
		S_{\delta_2}(u_0,\delta_2^{-1/p}\ASconstcaux,0) = S_{\delta_2}(u_0 + \delta_2^{-1/p}\ASconstcaux,0) - \delta_2^{-1/p}\ASconstcaux
	\]
	so that
	\[
		\|S_{\delta_2}(u_0 + \delta_2^{-1/p}\ASconstcaux,0) - (u_0 - \mathsf{P}_N\ASconstcaux^p + \delta_2^{-1/p}\ASconstcaux)\|_{L^2} < \tfrac 14 \epsilon.
	\]
	By continuity, there exists a radius $\rho > 0$ such that
	\[
	 	\|S_{\delta_2}(u,0) - (u_0 - \mathsf{P}_N\ASconstcaux^p + \delta_2^{-1/p}\ASconstcaux)\|_{L^2} < \tfrac 12 \epsilon
	\]
	for all $u$ with \[\|u - (u_0 + \delta_2^{-1/p}\ASconstcaux)\|_{L^2} < \rho.\]
	By the induction hypothesis, there exists $\tilde \ASc_1 \in C([0,\delta_1];\HH_1)$ with $0 < \delta_1 < \tfrac 13 T$ such that
	$
	 	\|S_{\delta_1}(u_0,\tilde \ASc_1) - (u_0 + \delta_2^{-1/p}\ASconstcaux)\|_{L^2} < \rho,
	$
	and therefore such that
	\begin{align*}
		\|S_{\delta_2}(S_{\delta_1}(u_0,\tilde \ASc_1), 0) - (u_0 - \mathsf{P}_N \ASconstcaux^p + \delta_2^{-1/p }\ASconstcaux)\|_{L^2} < \tfrac 12 \epsilon.
	\end{align*}
	Yet again by the induction hypothesis, there exists $\tilde \ASc_3 \in C([0,\delta_3];\HH_1)$ with $0 < \delta_3 < \tfrac 13 T$ such that
	\[
		\|S_{\delta_3}(S_{\delta_2}(S_{\delta_1}(u_0,\tilde \ASc_1), 0),\tilde\ASc_3) - (S_{\delta_2}(S_{\delta_1}(u_0,\tilde \ASc_1), 0) - \delta_2^{-1/p} \ASconstcaux)\|_{L^2} < \tfrac 14 \epsilon.
	\]
	Therefore, by the triangle inequality,
	\[
		\|S_{\delta_3}(S_{\delta_2}(S_{\delta_1}(u_0,\tilde \ASc_1), 0),\tilde\ASc_3) - (u_0 - \mathsf{P}_N \ASconstcaux^p)\|_{L^2} < \tfrac 34 \epsilon.
	\]
	We conclude that~\eqref{eq:appr-goal} holds with~$ \ASc \in C([0,\delta_1 + \delta_2 + \delta_3];\HH_1)$ a good enough continuous approximation of the function $\one_{[0,\delta_1)}\tilde\ASc_1 + \one_{[\delta_1+\delta_2, \delta_1+\delta_2+\delta_3]}\tilde\ASc_3(\,\cdot\,-(\delta_1+\delta_2))$.
	Note that $0 < \delta_1 + \delta_2 + \delta_3 < T$ by construction.
\end{proof}

\begin{proof}[Proof of Lemma~\ref{B1}]
	Fix~$\ASconstcaux, \ASconstc \in \HH_N$ and let $u(t) = S_t(u_0,\AScaux,\ASc)$ with the constant controls~$\AScaux(t) \equiv \ASconstcaux$ and~$\ASc(t) \equiv \ASconstc$. Also let
	\[
	 	w(t) :=   u_0+t(\ASconstc - {\mathsf P}_N\ASconstcaux^p) \qquad \text{ and } \qquad v(t):= u(\delta t)-w(t).
	\]
	Clearly, the fact that~$u$ solves~\eqref{E:GPDE2} with $u(0) = u_0$ implies that~$v$ solves
	\[
		\dot v(t)-\nu \delta\Delta (v(t)+w(t)+\de^{-1/p}\ASconstcaux)+ \delta {\mathsf P}_N F(v(t)+w(t)+\de^{-1/p}\ASconstcaux)- {\mathsf P}_N\ASconstcaux^p =\delta h
	\]
	with~$v(0)=0$. Taking the scalar product in~$L^2$ of this equation with~$v(t)$, applying the Cauchy--Schwarz   inequality, and dropping the arguments~$(t)$ for notational simplicity,  we get
	\begin{align}
	\frac12\frac \d {\d t}  \| v \|^2_{L^2}     &\le  \Big(\nu \delta \|\Delta w\|_{L^2}+    \nu \delta^{1-1/p} \| \Delta \ASconstcaux\|_{L^2} +   \delta\|h\|_{L^2} \\
		&\qquad\qquad +\|\delta {\mathsf P}_N F(v + w + \de^{-1/p}\ASconstcaux)- {\mathsf P}_N\ASconstcaux^p\|_{L^2}\Big)  \|v\|_{L^2} \nonumber\\& \le C_1 \left(\delta^{1-1/p}+ \|\delta {\mathsf P}_N F(v+w+\de^{-1/p}\ASconstcaux)- {\mathsf P}_N\ASconstcaux^p\|_{L^2}\right)  \| v \| _{L^2}\label{E:50}
	\end{align} for any $t\le 1$ and $\delta\le 1$. Using the assumption (i) and   the Young inequality, we obtain
	      \begin{align}
	 \|\delta {\mathsf P}_N F(v+w+\de^{-1/p}\ASconstcaux)- {\mathsf P}_N\ASconstcaux^p\|_{L^2}&\le  C_2\delta \left(\|v\|_{L^2}^p+\|w\|_{L^2}^p+ \delta^{-(p-1)/p} \|\ASconstcaux\|^{p-1}_{L^2}+1\right)\nonumber\\
	 &\le C_3\delta \left(\|v\|_{L^2}^p + \delta^{-(p-1)/p}  +1\right).\label{E:51}
	\end{align}
	  Combining~\eqref{E:50} and~\eqref{E:51},   we see that
	\begin{align}
	 \frac \d {\d t}  \| v(t) \|_{L^2}^2    \le     C_4 \delta^{1/p}\left(\|v(t)\|_{L^2}^{p+1}+1   \right). \label{E:52}
	\end{align}
	  Let us set $A_\delta:=C_4 \delta^{1/p} $  and
	\begin{equation}
		\Phi(t):= A_\delta + A_\delta\int_0^t  \|v(s)\|_{L^2}^{p+1} \dd s.     \label{E:53}
	\end{equation}
	Then,~\eqref{E:52} is equivalent to
	\[
		(\dot \Phi)^{2/(p+1)}\le A_\delta^{2/(p+1)}\Phi,
	\]
	and
	\[
		\frac{\dot \Phi}{\Phi^{(p+1)/2}}\le A_\delta.
	\]
	Integrating this inequality, we derive
	\[
		\Phi(t)\le  A_\delta \left(1-\frac{p-1}{2}A_\delta^{(p+1)/2} t\right)^{-2/(p-1)}
	\]
	for all $0 \leq t <1  \wedge T_*(\delta)$, where
	\[
		T_*(\delta):=    \left(\frac{p-1}{2}A_\delta^{(p+1)/2}\right)^{-1}.
	\]
	Because $T_*(\delta) \uparrow \infty$ monotonically as $\delta \downarrow 0$, there exists $\delta_0 > 0$ small enough that
	\begin{equation}\label{E:54}
		\Phi(t) \le 2 A_\delta
	\end{equation}
	for all $0 \leq t \leq 1$, whenever $0 < \delta \leq \delta_0$.
	Then, combining~\eqref{E:52}--\eqref{E:54}, we obtain
	$$
		\|v(1)\|_{L^2}^2 \le C_5 \delta^{1/p}
	$$
	for some constant $C_5$ independent of~$\delta$. Thus $v(1)\to 0$ as $\delta\to 0$,   which implies~\eqref{E:49}.
\end{proof}

\section{Some results from measure theory}\label{A:C}
\subsection{Maximal couplings}

Let $\XXXX,\YYYY$, and $\ZZZZ$ be Polish  spaces endowed with their Borel $\sigma$-algebras,~$\zzzz\in \ZZZZ\mapsto\mu(\zzzz,{\cdot\,}),\,\mu'(\zzzz,{\cdot\,})$ be two random probability measures on~$\XXXX$, and $F:\XXXX\to \YYYY$ be a measurable mapping.
We denote by~$F_*\mu(\zzzz,{\cdot\,})$ the image of~$\mu(\zzzz,{\cdot\,})$ under $F$ (similarly for $\mu'$).
The following lemma on the existence of maximal couplings is a particular case of Exercise~1.2.30.ii in~\cite{KuSh} (see the last section   of the book for a proof).

\begin{lemma}\label{lem:m-c}
	There is a probability space $(\Omega,\FF,\IP)$ and measurable mappings $\xi, \xi':\ZZZZ\times \Omega\to \XXXX$ such that the following two properties are satisfied:
\begin{enumerate}
\item[$\bullet$] {for all~$\zzzz \in \ZZZZ$, $(\xi(\zzzz,{\cdot\,}),\xi'(\zzzz,{\cdot\,}))$ is a coupling of~$\mu(\zzzz,{\cdot\,})$ and~$\mu'(\zzzz,{\cdot\,})$ in the sense that
\begin{equation}\label{eq:m-c}
 \xi(\zzzz,{\cdot\,})_* \IP = \mu(\zzzz,{\cdot\,}) \qquad \text{and} \qquad
 \xi'(\zzzz,{\cdot\,})_* \IP = \mu'(\zzzz,{\cdot\,});
\end{equation}}
\item[$\bullet$] for all~$\zzzz \in \ZZZZ$, $(F(\xi(\zzzz,{\cdot\,})),F(\xi'(\zzzz,{\cdot\,})))$ is a maximal coupling of~$F_*\mu(\zzzz,{\cdot\,})$ and~$F_*\mu'(\zzzz,{\cdot\,})$ in the sense that
\begin{equation}\label{eq:m-c-F}
\IP\left(\{\omega\in\Omega : F(\xi(\zzzz,\omega)) \neq F(\xi'(\zzzz,\omega))\}\right) = \|F_*\mu(\zzzz,{\cdot\,}) - F_*\mu'(\zzzz,{\cdot\,})\|_\textnormal{var}
\end{equation}and the random variables  $F(\xi(\zzzz,{\cdot\,}))$ and $F(\xi'(\zzzz,{\cdot\,}))$ conditioned on the event
$$
\left\{\omega \in \Omega : F(\xi(\zzzz,\omega))\neq F(\xi'(\zzzz,\omega))\right\}
$$ are independent.
\end{enumerate}
\end{lemma}

\subsection{Images of   measures under regular mappings}\label{S:6.3}

 Let $\XXXX$ be a compact metric space, $\YYYY$ and $\ZZZZ$ be  finite-dimensional spaces,   and  $F:\XXXX\times \ZZZZ\to~\!\!\YYYY$ be a continuous mapping.   The following  is a consequence   of a   more general result proved in   Theorem~2.4 in~\cite{Sh07} (see also Chapter~9 of~\cite{Bog}). In this simplified context in finite dimension, it can be proven directly from the implicit function theorem and a change of variable.

\begin{lemma}\label{lem:image-meas}
Assume that the mapping $F(\xxxx,{\cdot\,}):\ZZZZ\to \YYYY$ is   differentiable for any  $\xxxx\in \XXXX$, the derivative $D_\zzzz F$ is continuous on~$\XXXX\times \ZZZZ$, the image of the linear operator $(D_\zzzz F)(\hat\xxxx,\hat\zzzz)$ has full rank for some $(\hat\xxxx,\hat\zzzz)\in \XXXX\times \ZZZZ$, and~$\varrho \in\PP(\ZZZZ)$ is a measure possessing  a positive continuous density with respect to the Lebesgue measure on~$\ZZZZ$.~Then there is  a   continuous function $\psi : \XXXX\times \YYYY\to \R_+$ and a   number $r>0$    such that
\[
\psi(\hat x,F(\hat\xxxx,\hat\zzzz))>0,
\]
and
\[
(F_*(\xxxx, {\cdot\,})\varrho)(\d y) \ge \psi(\xxxx,\yyyy)  \d\yyyy
\]
(as measures on~$\YYYY$) for all $\xxxx\in B_\XXXX(\hat x,r)$.
\end{lemma}

\bibliographystyle{alpha-custom}
\bibliography{sde-harm}

\begin{thebibliography}{CEHRB18}

\bibitem[AK87]{AK87}
L.~Arnold and W.~Kliemann.
\newblock On unique ergodicity for degenerate diffusions.
\newblock {\em Stochastics}, 21(1):41--61, 1987.

\bibitem[AKSS07]{AKSS07}
A.~A. Agrachev, S.~Kuksin, A.~V. Sarychev, and A.~Shirikyan.
\newblock On finite-dimensional projections of distributions for solutions of
  randomly forced 2{D} {N}avier--{S}tokes equations.
\newblock {\em Ann. Inst. Henri Poincar\'e (B) Probab. Statist.},
  43(4):399--415, 2007.

\bibitem[AS05]{AS05}
A.~A. Agrachev and A.~V. Sarychev.
\newblock {N}avier--{S}tokes equations: controllability by means of low modes
  forcing.
\newblock {\em J. Math. Fluid Mech.}, 7(1):108--152, 2005.

\bibitem[AS08]{AS-2008}
A.~A. Agrachev and A.~V. Sarychev.
\newblock Solid controllability in fluid dynamics.
\newblock In {\em Instability in {M}odels {C}onnected with {F}luid {F}lows.
  {I}}, volume~6 of {\em Int. Math. Ser.}, pages 1--35. Springer, New York,
  2008.

\bibitem[BC09]{BC09}
A.~Baule and E.~G.~D. Cohen.
\newblock Fluctuation properties of an effective nonlinear system subject to
  poisson noise.
\newblock {\em Phys. Rev. E}, 79:030103, 2009.

\bibitem[Bog10]{Bog}
V.~I. Bogachev.
\newblock {\em Differentiable measures and the {M}alliavin calculus}, volume
  164 of {\em Mathematical Surveys and Monographs}.
\newblock Amer.\ Math.\ Soc., 2010.

\bibitem[CEHRB18]{CEHRB}
N.~Cuneo, J.-P. Eckmann, M.~Hairer, and L.~Rey-Bellet.
\newblock Non-equilibrium steady states for networks of oscillators.
\newblock {\em Electron. J. Probab.}, 23(Paper No.~55):1--28, 2018.

\bibitem[Cor07]{Cor}
J.-M. Coron.
\newblock {\em Control and nonlinearity}, volume 136 of {\em Mathematical
  Surveys and Monographs}.
\newblock Amer.\ Math.\ Soc., 2007.

\bibitem[DPZ96]{dPZa}
G.~Da~Prato and J.~Zabczyk.
\newblock {\em Ergodicity for infinite dimensional systems}, volume 229 of {\em
  London Math. Soc. Lecture Notes Series}.
\newblock Cambridge University Press, 1996.

\bibitem[EH00]{EH00}
J.-P. Eckmann and M.~Hairer.
\newblock Non-equilibrium statistical mechanics of strongly anharmonic chains
  of oscillators.
\newblock {\em Commun. Math. Phys.}, 212(1):105--164, 2000.

\bibitem[EPRB99a]{EPR99b}
J.-P. Eckmann, C.-A. Pillet, and L.~Rey-Bellet.
\newblock Entropy production in nonlinear, thermally driven {H}amiltonian
  systems.
\newblock {\em J. Stat. Phys.}, 95(1-2):305--331, 1999.

\bibitem[EPRB99b]{EPR99a}
J.-P. Eckmann, C.-A. Pillet, and L.~Rey-Bellet.
\newblock Non-equilibrium statistical mechanics of anharmonic chains coupled to
  two heat baths at different temperatures.
\newblock {\em Commun. Math. Phys.}, 201(3):657--697, 1999.

\bibitem[FKM65]{FKM65}
G.~W. Ford, M.~Kac, and P.~Mazur.
\newblock Statistical mechanics of assemblies of coupled oscillators.
\newblock {\em J. Math. Phys.}, 6:504--515, 1965.

\bibitem[Gri75]{Gr75}
D.~Griffeath.
\newblock A maximal coupling for {M}arkov chains.
\newblock {\em Z. Wahrscheinlichkeitstheorie und Verw. Gebiete}, 31:95--106,
  1974/75.

\bibitem[Hai11]{Ha11}
M.~Hairer.
\newblock On {M}alliavin's proof of {H}\"{o}rmander's theorem.
\newblock {\em Bull. Sci. Math.}, 135(6-7):650--666, 2011.

\bibitem[JK85]{JK-1985}
V.~Jurdjevic and I.~Kupka.
\newblock Polynomial control systems.
\newblock {\em Math. Ann.}, 272(3):361--368, 1985.

\bibitem[JP97]{JP97}
V.~Jak{\v{s}}i{\'c} and C.-A. Pillet.
\newblock Ergodic properties of the non-{M}arkovian {L}angevin equation.
\newblock {\em Lett. Math. Phys.}, 41(1), 1997.

\bibitem[JP98]{JP98}
V.~Jak{\v{s}}i{\'c} and C.-A. Pillet.
\newblock Ergodic properties of classical dissipative systems {I}.
\newblock {\em Acta Math.}, 181(2):245--282, 1998.

\bibitem[JPS17]{JPS17}
V.~Jak{\v{s}}i{\'{c}}, C.-A. Pillet, and A.~Shirikyan.
\newblock Entropic fluctuations in thermally driven harmonic networks.
\newblock {\em J. Stat. Phys.}, 166(3):926--1015, Feb 2017.

\bibitem[Jur97]{J-1997}
V.~Jurdjevic.
\newblock {\em Geometric control theory}, volume~52 of {\em Cambridge Studies
  in Advanced Mathematics}.
\newblock Cambridge University Press, 1997.

\bibitem[Kha12]{H-2012}
R.~Khasminskii.
\newblock {\em Stochastic stability of differential equations}, volume~66 of
  {\em Stochastic Modelling and Applied Probability}.
\newblock Springer, Heidelberg, second edition, 2012.
\newblock With contributions by G. N. Milstein and M. B. Nevelson.

\bibitem[KS12]{KuSh}
S.~Kuksin and A.~Shirikyan.
\newblock {\em Mathematics of two-dimensional turbulence}, volume 194 of {\em
  Cambridge Tracts in Mathematics}.
\newblock Cambridge University Press, 2012.

\bibitem[MG12]{MG12}
W.~A.~M. Morgado and T.~Guerreiro.
\newblock A study on the action of non-{G}aussian noise on a {B}rownian
  particle.
\newblock {\em Physica A Stat. Mech. Appl.}, 391(15):3816--3827, 2012.

\bibitem[MQSP11]{MQ+11}
W.~A.~M. Morgado, S.~M.~D. Queir{\'o}s, and D.~O. Soares-Pinto.
\newblock On exact time averages of a massive {P}oisson particle.
\newblock {\em J. Stat. Mech. Theory Exp.}, 2011(06):P06010, 2011.

\bibitem[MT93]{MT1993}
S.~P. Meyn and R.~L. Tweedie.
\newblock {\em Markov {C}hains and {S}tochastic {S}tability}.
\newblock Springer-Verlag, 1993.

\bibitem[MT12]{MeTw}
S.~P. Meyn and R.~L. Tweedie.
\newblock {\em {M}arkov chains and stochastic stability}.
\newblock Communications and Control Engineering Series. Springer Science \&
  Business Media, 2012.

\bibitem[Ner08]{Ne08}
V.~Nersesyan.
\newblock Polynomial mixing for the complex {G}inzburg-{L}andau equation
  perturbed by a random force at random times.
\newblock {\em J. Evol. Equ.}, 8(1):1--29, 2008.

\bibitem[Ner20]{N-2019}
V.~Nersesyan.
\newblock {Approximate controllability of nonlinear parabolic PDEs in arbitrary
  space dimension}.
\newblock {\em Math. Control Relat. Fields}, To appear, 2020.

\bibitem[Nua06]{nualart2006}
D.~Nualart.
\newblock {\em {The {M}alliavin Calculus and Related Topics}}.
\newblock Springer-Verlag, Berlin, 2006.

\bibitem[Raq19]{Ra18}
R.~Raqu{\'e}pas.
\newblock A note on {H}arris' ergodic theorem, controllability and
  perturbations of harmonic networks.
\newblock {\em Ann. Henri Poincar{\'e}}, 20(2):605--629, 2019.

\bibitem[RBT02]{RBT02}
L.~Rey-Bellet and L.~E. Thomas.
\newblock Exponential convergence to non-equilibrium stationary states in
  classical statistical mechanics.
\newblock {\em Commun. Math. Phys.}, 225(2):305--329, 2002.

\bibitem[Shi07]{Sh07}
A.~Shirikyan.
\newblock Qualitative properties of stationary measures for three-dimensional
  {N}avier-{S}tokes equations.
\newblock {\em J. Funct. Anal.}, 249(2):284--306, 2007.

\bibitem[Shi08]{Sh08}
A.~Shirikyan.
\newblock Exponential mixing for randomly forced partial differential
  equations: method of coupling.
\newblock In {\em Instability in models connected with fluid flows. {II}},
  volume~7 of {\em Int. Math. Ser. (N. Y.)}, pages 155--188. Springer, New
  York, 2008.

\bibitem[Shi17]{Sh17}
A.~Shirikyan.
\newblock Controllability implies mixing. {I}. {C}onvergence in the total
  variation metric.
\newblock {\em Uspekhi Mat. Nauk}, 72(5(437)):165--180, 2017.

\bibitem[Shi18]{MR3777005}
A.~Shirikyan.
\newblock Control theory for the {B}urgers equation: {A}grachev-{S}arychev
  approach.
\newblock {\em Pure Appl. Funct. Anal.}, 3(1):219--240, 2018.

\bibitem[SL77]{LS77}
H.~Spohn and J.~L. Lebowitz.
\newblock Stationary non-equilibrium states of infinite harmonic systems.
\newblock {\em Commun. Math. Phys.}, 54(2):97--120, 1977.

\bibitem[TC09]{TC09}
H.~Touchette and E.~G.~D. Cohen.
\newblock Anomalous fluctuation properties.
\newblock {\em Phys. Rev. E}, 80:011114, Jul 2009.

\bibitem[Tro77]{Tr77}
M.~M. Tropper.
\newblock Ergodic and quasideterministic properties of finite-dimensional
  stochastic systems.
\newblock {\em J. Stat. Phys.}, 17(6):491--509, 1977.

\end{thebibliography}

\end{document}